\pdfoutput=1
\documentclass{wsbart/wsbart}

\usepackage[utf8]{inputenc}
\usepackage[T1]{fontenc}
\usepackage{lmodern}

\usepackage[maxnames=99,style=trad-plain,backend=biber,block=space,isbn=false]
{biblatex}
\newcommand{\citelink}[2]{\hyperlink{cite.\therefsection @#1}{#2}}
\DeclareRedundantLanguages{English}{english}

\usepackage{amsmath,amssymb,amsthm,stmaryrd,bbm}
\usepackage{mathtools}

\theoremstyle{plain}
\newtheorem{thm}{Theorem}
\newtheorem{cor}[thm]{Corollary}
\newtheorem{lem}[thm]{Lemma}
\newtheorem{prop}[thm]{Proposition}
\theoremstyle{definition}

\newtheorem{assu}{Assumption}
\newtheorem{exm}{Example}
\theoremstyle{remark}
\newtheorem{rem}{Remark}

\newcommand{\po}{\left(}
\newcommand{\pf}{\right)}

\newcommand{\R}{\mathbb R}

\newcommand{\N}{\mathbb N}
\newcommand{\dd}{\mathop{}\!\textnormal{d}}
\newcommand{\na}{\nabla}
\newcommand{\1}{\mathbbm{1}}
\newcommand{\refc}{\text{rc}}
\newcommand{\sync}{\text{sc}}
\DeclareMathOperator{\Expect}{\mathbb E}
\DeclareMathOperator{\Law}{Law}

\title{Logarithmic Sobolev inequalities\\for non-equilibrium steady states}
\author[1]{Pierre Monmarché}
\author[2]{Songbo Wang}
\affil[1]{LJLL \& LCT, Sorbonne Université, Paris, France}
\affil[2]{CMAP, École polytechnique, IP Paris, Palaiseau, France}

\begin{document}
\maketitle

\begin{abstract}
We consider two methods to establish log-Sobolev inequalities for the invariant measure of a diffusion process when its density is not explicit and the   curvature is not positive everywhere.
In the first approach,
based on the Holley--Stroock and Aida--Shigekawa perturbation arguments
[\citelink{HolleyStroockLSI}%
{\textit{J.\ Stat.\ Phys.}, 46(5-6):1159--1194, 1987},
\citelink{AidaShigekawaLSI}%
{\textit{J.\ Funct.\ Anal.}, 126(2):448--475, 1994}],
the control on the (non-explicit) perturbation is obtained by stochastic
control methods, following the comparison technique introduced by Conforti
[\citelink{ConfortiCouplReflControlDiffusion}%
{\textit{Ann.\ Appl.\ Probab.}, 33(6A):4608--4644, 2023}].
The second method combines the Wasserstein-$2$ contraction method, used in
[\citelink{MonmarcheHighTemperature}%
{\textit{Ann.\ Henri Lebesgue}, 6:941–973, 2023}]
to prove a Poincaré inequality in some non-equilibrium cases,
with Wang's hypercontractivity results
[\citelink{WangLSIRicHess}%
{\textit{Potential Anal.}, 53(3):1123--1144, 2020}].
\end{abstract}

\section{Introduction}\label{sec:intro}

\subsection{Overview}

A probability measure $\mu$ on $\R^d$ is said to satisfy a logarithmic Sobolev inequality (LSI) with constant $C_\textnormal{LS}>0$
if for all smooth and compactly supported function $f$ on $\R^d$
with $\int f^2 \dd \mu=1$, we have
\[\int_{\R^d} f^2 \ln f^2  \dd \mu
\leqslant 2C_\textnormal{LS} \int_{\R^d} |\na f|^2  \dd \mu\,.\]
It is related to the long-time convergence of diffusion processes and concentration inequalities, see \cite{BGLMarkov} and references therein for general considerations on this topic. The main question addressed here is to establish such an LSI in cases where $\mu$ has no explicit density but is defined as the invariant measure of a diffusion process $(Z_t)_{t\geqslant 0}$ satisfying
\begin{equation}\label{eq:EDS}
\dd Z_t = b(Z_t) \dd t + \sigma \dd B_t\,,
\end{equation}
where $b\in\mathcal C^1(\R^d,\R^d)$, $\sigma$ is a constant    matrix and $B$ is a $d$-dimensional Brownian motion. Among other applications,  this is motivated by non-equilibrium statistical physics models, such as \cite{IOSConvergence,IOSThermo,CuneoEckmann}.
In this literature, these non-explicit invariant measures are referred to as Non-Equilibrium Steady States (NESS).

Many criteria to establish LSI are known, but most of them require an explicit expression for $\mu$ (for instance in order to use bounded or Lipschitz perturbation arguments) or the process \eqref{eq:EDS} to be reversible with respect to $\mu$ (for instance the Lyapunov-based results of \cite{BCGRate,CGZPoincare}). In fact, denoting by
\[\mathcal L = b\cdot\na +  \Sigma :\na^2\]
the generator of \eqref{eq:EDS} (where $\Sigma = \sigma \sigma^T\!/2$ and $\Sigma :\na^2 = \sum_{i,j} \Sigma_{ij}\partial_{z_i}\partial_{z_j}$),  some arguments based on reversibility (such as those of  \cite{BCGRate,CGZPoincare}) may sometimes be extended to non-reversible cases
when the dual $\mathcal L^*$ of $\mathcal L$ in $L^2(\mu)$ is known. Since
\[\mathcal L^* = (2 \Sigma\na \ln \mu -b)\cdot\na + \Sigma :\na^2\,,\]
knowing $\mathcal L^*$ requires an explicit expression of $\mu$.
A notable exception is the use of Bakry--Émery curvature conditions: if there exists $\rho>0$ such that, for all $x$, $y\in\R^d$,
\begin{equation}\label{eq:contraction}
\bigl(b(x)-b(y)\bigr) \cdot (x-y) \leqslant -\rho |x-y|^2\,,
\end{equation}
then the diffusion \eqref{eq:EDS} admits a unique invariant measure that satisfies an LSI with constant $|\Sigma|/\rho$, see \cite{CattiauxGuillinSemiLogConcave,MonmarcheAlmostSure}, even when the process is neither reversible nor elliptic. However, such a contraction condition is very restrictive.
If \eqref{eq:contraction} holds only for $x$, $y$ outside some compact set, we can decompose $b=b_0+b_1$ where $b_0$ satisfies a similar condition on the whole $\R^d$ and $b_1$ is compactly supported. Then, we know that the invariant measure $\mu_0$ of the process with generator
$\mathcal L_0 = b_0\cdot\na +\Sigma:\na^2$ satisfies an LSI,
but to the best of our knowledge it is not known how to transfer the result to $\mu$ in this general case.

In this work, we will consider two cases:
\begin{itemize}
\item In the first one, $\mu$ is a perturbation of an explicit measure $\mu_0$,
invariant for $\mathcal L_0 = b_0\cdot\na +\Sigma:\na^2$ for some $b_0$.
Our method relies
on the bounded perturbation result of Holley and Stroock \cite{HolleyStroockLSI}
and the Lipschitz perturbation result of Aida and Shigekawa \cite{AidaShigekawaLSI}.
In other words, the key point is to prove that $\ln(\mu/\mu_0)$
is the sum of a bounded function and a Lipschitz function.
This is done by seeing this quantity as the long-time limit of the solution
of a (parabolic) Hamilton--Jacobi--Bellman (HJB) equation
and using a stochastic control representation for the solution together with a coupling argument,
following the method introduced by Conforti in \cite{ConfortiCouplReflControlDiffusion}.
This approach is applied to the elliptic case and to a non-elliptic kinetic case.
\item In the second one, we consider the high-diffusivity elliptic framework of \cite{MonmarcheHighTemperature},
namely \eqref{eq:contraction} holds for every $y$ once $x$ lies outside some compact set
and $\sigma = \bar \sigma \mathrm{Id}$ where $\bar \sigma>0$ is large enough. Under these conditions, on the one hand, it is proven in \cite{MonmarcheHighTemperature} that $\mu$ satisfies a Poincaré inequality, using the large-time contraction of the Wasserstein-2 distance along the diffusion semi-group. On the other hand, as established by Wang in \cite{WangLSIRicHess} (in the reversible case but we will see that the proof applies without any change in the non-reversible case), the semi-group is hypercontractive, which implies a so-called defective LSI (which is well known in the reversible case and turns out to be true in general).
The Poincaré inequality together with the defective LSI is equivalent to the LSI.
\end{itemize}

In fact, thanks to the powerful  \cite[Corollary 1.2]{WangCriteriaSpectralGap}, the defective LSI alone is already equivalent to a tight LSI for irreducible diffusion processes. This has been used in subsequent works \cite{WangFunctionalInequalitiesNonSymmetric,WangExponentialContraction,HKRLSIDecoupled} for both elliptic and kinetic processes. However this argument is non-constructive and thus does not provide explicit  constants, similarly to the tightening argument based on weak Poincaré inequalities in \cite[Proposition 1.3]{RWWeakPoincare}. This is in our contrast to our approach, as illustrated in  \cite{lsiut} which is based on the present work.

The rest of this work is organized as follows. The results are stated in the remainder of Section~\ref{sec:intro}. The results based on perturbation,
Theorems~\ref{thm:elliptic} and \ref{thm:kinetic}, are proven in Section~\ref{sec:proofs_Lipschitz}. Section~\ref{sec:defective_proof} is devoted to the proofs for the defective LSI. A coupling construction for the kinetic Langevin process, used in the proof of Theorem~\ref{thm:kinetic}, is postponed to Appendix~\ref{sec:couplage_cinetique}.

\subsection{Perturbation approach: the elliptic case}\label{sec:Lipschitz_elliptic}

In the elliptic case where $\Sigma=\mathrm{Id}$, we get the following,
proven in Section~\ref{sec:Lipschitz_elliptic_proof}.

\begin{thm}\label{thm:elliptic}
Assume that $\Sigma =  \mathrm{Id}$ and $b=b_0 + b_1$
for some $b_0$, $b_1\in\mathcal C^1(\R^d,\R^d)$ with bounded derivatives
such that the generator $\mathcal L_0 = b_0\cdot \na + \Delta $
admits a unique $C^2$ and positive
invariant probability density $\mu_0$
satisfying an LSI with constant $C_{0}>0$.
Write $\tilde b \coloneqq 2\na \ln \mu_0 - b$
and $\varphi \coloneqq -\na \cdot b_1 + b_1 \cdot \na \ln\mu_0 $.
Assume that there exist $L$, $R$, $M^\varphi$, $L^\varphi \geqslant 0$
and $\rho > 0$ such that
$\varphi = \varphi_1 + \varphi_2$ with $\varphi_1$ being \(M^\varphi\)-bounded
and $\varphi_2$ being $L^\varphi$-Lipschitz,
and for all $x$, $y\in\R^d$,
\begin{equation}\label{eq:condition_contraction}
\bigl(\tilde b(x)-\tilde b(y)\bigr) \cdot (x-y) \leqslant \begin{cases}
- \rho |x-y|^2 & \text{if $|x-y|\geqslant R$},\\
L|x-y|^2 & \text{otherwise}.
\end{cases}
\end{equation}
Finally, assume that the law of $Z_t$ solving  \eqref{eq:EDS} converges weakly for all initial condition as $t\rightarrow \infty$ to  a unique invariant measure
$\mu$ on $\R^{d}$. Then $\mu$ satisfies an LSI with constant
$C_\textnormal{LS} = C_\textnormal{LS}(C_0, \rho, L, R, d;
M^\varphi, L^\varphi)$.
\end{thm}

Notice that, when $\Sigma=\mathrm{Id}$, the carré du champ operator
$\Gamma(f) = \frac12 \mathcal L(f^2) - f \mathcal Lf$ is equal to $|\nabla f|^2$
and is the same for the dual operator $\mathcal L^*$.
In particular, the LSI is equivalent to the constant-rate decay of the relative entropy,
\begin{equation}
\label{eq:entropy_decay}
\forall \nu\ll \mu \,,\qquad  \mathcal H(P_t^* \nu | \mu) \leqslant e^{-t/C_\textnormal{LS}} \mathcal H(\nu | \mu)\,,
\end{equation}
where $P_t = \exp(t \mathcal L)$ is the semi-group generated by $\mathcal L$,
and the relative entropy $\mathcal H$ is defined by $\mathcal H(\nu | \mu) = \int \ln(\dd\nu/\!\dd\mu) \dd \nu $ (see e.g. \cite[Theorem 5.2.1]{BGLMarkov}; reversibility is not used in the proof).

\begin{exm}
Consider on $\R^2$ the SDE
\[\dd X_t = \bigl( f(|X_t|)  X_t^\perp - X_t - \na V(X_t) \bigr) \dd t
+ \sqrt{2}\dd B_t\,,\]
where $(u,v)^\perp = (v,-u)$, $f\in\mathcal C^1(\R_+,\R)$  and $V\in\mathcal C^2(\R^2)$. We can decompose the drift $b=b_0+b_1$ as
\begin{alignat}{2}
\label{eq:perp}
b_0(x) &=  x^\perp f(|x|) -x\,,&\qquad b_1(x) &= -\na V(x)\,, \\
\intertext{or, alternatively,}
\label{eq:perp_alternative}
b_0(x) &= -\na V(x)-x \,,&b_1(x) &= f(|x|) x^\perp  \,.
\end{alignat}
In the first case \eqref{eq:perp}, the invariant measure of $b_0\cdot\na +\Delta$ is the standard Gaussian measure $\mu_0(x)= (2\pi)^{-d/2}\exp(-|x|^2\!/2)$ and, using the notations of Theorem~\ref{thm:elliptic},
\[\tilde b(x) = -  f(|x|)  x^\perp -x + \na V(x)\,,\qquad \varphi(x) = \Delta V(x) + \na V(x) \cdot x  \,.\]
For instance, if $f(|x|)$ is constant for $|x|$ large enough, then the rotating part does not intervene in the condition \eqref{eq:condition_contraction} outside some compact set, which means that this condition is satisfied as soon as
$(x-y)\cdot \bigl(\na V(x) - \na V(y)\bigr) \leqslant \eta |x-y|^2$
outside some compact set for some $\eta<1$. In this situation, ergodicity for \eqref{eq:EDS} is easily shown using Harris Theorem. Then Theorem~\ref{thm:elliptic} applies as soon as $\varphi$ is Lipschitz, which is for instance the case if $V$ is compactly supported (which implies the previous condition).
Notice that here we do not use the fact
that the perturbative term $b_1 = -\na V$ is a gradient.

If, alternatively, we use the decomposition \eqref{eq:perp_alternative}, then $\mu_0$ is the probability density proportional to $\exp \bigl(-|x|^2\!/2 - V(x)\bigr)$ and
\[\tilde b(x) = -   f(|x|) x^\perp  -x - \na V(x)\,,\qquad \varphi(x) =   f(|x|)   x^\perp  \cdot   \na V(x) \,.  \]
Then, Theorem~\ref{thm:elliptic} applies for instance if $f$ is compactly supported
and $(x-y)\cdot \bigl(\na V(x) - \na V(y)\bigr) \geqslant -\eta |x-y|^2$
for some $\eta<1$ outside some compact set.
\end{exm}

\begin{exm}
Let us check how  Theorem~\ref{thm:elliptic}  reads in the classical reversible case, namely taking
\[b_0(x) = -\na U(x)\,,\qquad b_1(x) = -\na W(x)\]
for some $U$, $W\in\mathcal C^2(\R^d)$, so that $\mu_0\propto e^{-U}$, $\mu\propto e^{-U-W}$ and, with the notations of Theorem~\ref{thm:elliptic},
\[\tilde b(x) = -\na U(x) + \na W(x) \,,\qquad \varphi(x) = \Delta  W(x) +\na U(x)\cdot\na W(x)  \,.  \]
Hence, the conditions in Theorem~\ref{thm:elliptic} does not seem to be similar to those of classical perturbation results in the reversible case. However, for instance, if we take $U(x)=|x|^2\!/2$ outside some compact set, then to get that  $x\mapsto \na U(x)\cdot\na W(x)$ is Lipschitz one will typically require $\na W$ to be bounded, in which case the LSI for $\mu$ follows from \cite{AidaShigekawaLSI}.
\end{exm}

\begin{exm}
We now consider a non-linear McKean-Vlasov equation on $\R^d$:
\begin{equation}\label{eq:McKeanVlasov}
    \partial_t \mu_t = \nabla \cdot \bigl(
(\na V + \lambda b_{\mu_t}) \mu_t + \na \mu_t\bigr)
\end{equation}
where $x\mapsto \mu_t(x)$ is a probability density on $\R^d$ (which we identify
with the corresponding probability measure), $V\in\mathcal C^2(\R^d)$ is a
confining potential, $\lambda\in\R$ encodes the non-linearity amplitude,
$\na\cdot$ stands for the divergence operator and, for all probability measure
$\nu$, $b_\nu\in\mathcal C^1(\R^d)$. If $\mu_*$ is a stationary measure for
\eqref{eq:McKeanVlasov}, it is the invariant measure of the diffusion process
with generator $\mathcal L_{\mu_*}$ where
\[\mathcal L_{\mu} = -(\na V+\lambda b_{\mu})\cdot \na + \Delta  \,. \]
Among other examples of interest where, for a given $\mu$, $b_\mu$ is not the gradient of some potential, we can mention the competition models considered in \cite{LuTwoScale}, where $x=(x_1,x_2)\in\R^{2p}$ and
\begin{equation}\label{eq:competition}
b_{\mu}(x_1,x_2) = \begin{pmatrix}
    \int_{\R^p} \na_{x_1} K(x_1,y_2) \mu(y_1,y_2)\dd y_1\dd y_2 \\
    -\int_{\R^p} \na_{x_2} K(y_1,x_2) \mu(y_1,y_2)\dd y_1\dd y_2
\end{pmatrix}
\end{equation}
for some $K\in\mathcal C^2(\R^{2p})$. In other words, the population is divided
in two types of individuals, the first type (resp.\ second) tends to maximize
(resp.\ minimize) its value of $K$ averaged with respect to the population distribution of the other type.

In order to get an LSI for $\mu_*$,  for instance, it is straightforward to check that the assumptions of Theorem~\ref{thm:elliptic} are met (with $\mu_0 \propto e^{-V}$) under the following condition:
\begin{assu}\label{assu:McKeanVlasov}
The potential $V$ is strictly convex outside a compact, and its hessian is bounded.
There exists $L'$, $C'>0$ such that, for all $\mu$, $b_\mu$ is $L'$-Lipschitz and for all $x\in\R^d$,
    \begin{equation}\label{eq:condMcKV}
         b_{\mu}(x) \leqslant C' \frac{1+\int_{\R^d} |y| \mu(\dd y)}{1+|x|}\,.
    \end{equation}
    \end{assu}
    In particular, the condition \eqref{eq:condMcKV} is used to get that, for a given $\mu$ with finite expectation, $\varphi$ is bounded (as $|\na \ln\mu_0(x)| = |\na V(x)|\leqslant C(1+|x|)$ for some $C>0$). The condition that $\na^2 V$ is bounded can be lifted if \eqref{eq:condMcKV} is replaced by a stronger decay of $b_\mu(x)$.

For instance, in the case \eqref{eq:competition}, the condition \eqref{eq:condMcKV} holds when $|\na_x K(x,y)| \leqslant C /(1+|x-y|^2)$ for some constant $C>0$ (and similarly for $\na_y K$). Indeed, then, considering the first coordinate of \eqref{eq:competition} (the second one being similar), we bound
\begin{align*}
\left|\int_{\R^p} \na_{x_1} K(x_1,y_2) \mu(y_1,y_2)\dd y_1\dd y_2\right| &\leqslant
C\int_{\R^p} \frac{1}{1+|x-y_2|^2} \mu(y_1,y_2)\dd y_1\dd y_2  \\
& \leqslant \frac{C}{1+|x|} + C \mathbb P_\mu \po |Y_2| \geqslant |x| - \sqrt{|x|}\pf
\end{align*}
and the Markov inequality concludes. The fact that a stationary solution of
\eqref{eq:McKeanVlasov} has a finite expectation is implied by
Assumption~\ref{assu:McKeanVlasov} for $\lambda$ small enough, as e.g.\ $w(x) =|x|^2$ is a Lyapunov function for \eqref{eq:McKeanVlasov}.

Contrary to the linear case, obtaining an LSI for $\mu_*$ is however not
sufficient to get the exponential convergence of the solution of
\eqref{eq:McKeanVlasov} toward $\mu_*$, as, even under
Assumption~\ref{assu:McKeanVlasov}, several stationary solutions may exist
\cite{HerrmannTugaut}. However, this is sufficient to conclude in the weak
interaction regime (i.e.\ when $\lambda$ is small enough), provided the interaction drift is Lipschitz in terms of the non-linearity:
\begin{equation}\label{eq:McKV_Lipschitz}
\exists B>0\ \text{s.t.}\ \forall \nu,\nu', \qquad \|b_{\nu}-b_{\nu'}\|_\infty  \leqslant B \mathcal W_2(\nu,\nu')\,,
\end{equation}
where the $\mathcal W_2$-Wasserstein distance between two probability measures
$\nu$, $\nu'$ is defined as
\[\mathcal W_2(\nu,\nu') = \underset{\pi \in \mathcal C(\nu,\nu')} \inf\po \int_{(\R^d)^2} |x-x'|^2 \pi(\dd x,\dd x')\pf^{1/2}\,,\]
with $\mathcal C(\nu,\nu')$ the set of probability measures on $(\R^d)^2$ with marginal $\nu$ and $\nu'$. Indeed, in that case, by a classical computation,
\begin{align*}
    \partial_t \mathcal H\po \mu_t | \mu_*\pf &=
    \int_{\R^d} \partial_t(\ln \mu_t) \mu_t + \ln\frac{\mu_t}{\mu_*} \partial_t \mu_t\\
    & = \partial_t \int_{\R^d}\mu_t + \int_{\R^d} \mathcal L_{\mu_t}\po \ln \frac{\mu_t}{\mu_*}\pf \mu_t\\
    & = \int_{\R^d} \mathcal L_{\mu_*}\po \ln \frac{\mu_t}{\mu_*}\pf \mu_t +
\int_{\R^d} (\mathcal L_{\mu_t}-\mathcal L_{\mu_*})\po \ln \frac{\mu_t}{\mu_*}\pf \mu_t\\
    & = -\int_{\R^d}\left|\na \ln \frac{\mu_t}{\mu_*}\right|^2 \mu_t + \lambda \int_{\R^d}\na \ln \frac{\mu_t}{\mu_*} \cdot (b_{\mu_*}-b_{\mu_t}) \mu_t \\
    & \leqslant -\frac12 \int_{\R^d}\left|\na \ln \frac{\mu_t}{\mu_*}\right|^2 \mu_t + \frac12 B^2\lambda^2 \mathcal W_2^2(\mu_t,\mu_*)\,,
\end{align*}
where we used   Cauchy--Schwarz and \eqref{eq:McKV_Lipschitz}. Now, the LSI satisfied by $\mu_*$ implies the Talagrand inequality
\[ \mathcal W_2^2 (\mu_t,\mu_*) \leqslant C \mathcal H\po \mu_t | \mu_*\pf\]
(where $C$ is the LSI constant of $\mu_*$), see \cite{OttoVillani}. As  a consequence, using that the LSI constant of $\mu_*$ is uniformly bounded over small values of $\lambda$ (since Theorem~\ref{thm:elliptic} can be applied with the same constants $L^\varphi$ and $M^\varphi$ for all values of $\lambda\in[0,\lambda_0]$ for any $\lambda_0>0$), we get that
\[
    \partial_t \mathcal H\po \mu_t | \mu_*\pf
    \leqslant  -\varepsilon \mathcal H\po \mu_t | \mu_*\pf
\]
for some $\varepsilon>0$ for $\lambda$ small enough. As a conclusion, we obtain that $\mu_*$ is the unique stationary solution of \eqref{eq:McKeanVlasov} and globally attractive.

Notice that \eqref{eq:McKV_Lipschitz} holds for the model
\eqref{eq:competition} as soon as $\na^2 K$ is bounded, since in that case,
\begin{multline*}
  \left| \int_{\R^p} \na_{x_1} K(x_1,y_2) \nu(y_1,y_2)\dd y_1\dd y_2  - \int_{\R^p} \na_{x_1} K(x_1,y_2) \nu'(y_1,y_2)\dd y_1\dd y_2 \right | \\
  \leqslant \|\na^2 K\|_\infty \int_{(\R^p)^2} |y_2-y_2'| \pi(\dd y_1\dd
y_2,\dd y_1\dd y_1')\,,
\end{multline*}
where $\pi$ is any coupling of $\nu$ and $\nu'$, so that conclusion follows by
Cauchy--Schwarz and taking the infimum over all couplings (the second coordinate of $b_\nu-b_{\nu'}$ being treated similarly).
\end{exm}

\subsection{Perturbation approach: the kinetic case}

We consider in this section a non-equilibrium Langevin diffusion $Z=(X,V)$ on $\R^d\times \R^d$ solving
\begin{equation}\label{eq:kinetic}
\left\{\begin{array}{rcl}
\dd X_t &= & V_t\dd t\\
\dd V_t &= & -\na U(X_t) \dd t + G(X_t,V_t)\dd t - \gamma V_t \dd t + \sqrt{2 \gamma  } \dd B_t
\end{array}\right.
\end{equation}
for some $\gamma >0$, $U\in\mathcal C^2(\R^d)$, $G\in\mathcal C^1(\R^{2d},\R^d)$, as studied in \cite{IOSConvergence,MonmarcheRamilOverdamped}. The particular case where $G$ depends only on $V$ correspond to non-linear friction models, see e.g. \cite{MarchesoniNonlinear,KPNonGaussian}.

Contrary to the elliptic case, now, the LSI is not equivalent to the entropy decay \eqref{eq:entropy_decay}. However, from the LSI, the (hypocoercive) decay of the entropy along $(P_t)_{t\geqslant0}$
(rather than $(P_t^*)^{\phantom{*}}_{t\geqslant 0}$)
can be obtained applying Theorems 9 and 10 of \cite{MonmarcheGeneralized} even without the explicit knowledge of the invariant measure.

\begin{thm}\label{thm:kinetic}
Assume that $e^{-U}$ is integrable and that the probability measure with density proportional to $e^{-U}$ satisfies an LSI with constant $C_0$. Let
\[\varphi(x,v) =  - \na_v G(x,v)+  G(x,v) \cdot v\,.\]
Assume that $\varphi$ is $L^\varphi$-Lipschitz
and the drift writes
\[
-\nabla U(x) + G(x, -v) = - K x + g(x,v)
\]
for a positive-definite matrix $K$ whose smallest eigenvalue is \(k > 0\),
and a function $g : \R^{2d} \to \R$ satisfying
\[
|g(x,v) - g(x',v')| \leqslant \begin{cases}
L_1 |z - z'| & \text{if $|x - x'| + |v - v'| \leqslant R$}, \\
L_2 |z - z'| & \text{otherwise},
\end{cases}
\]
where $|z - z'| = \sqrt{ |x - x'|^2 + |v - v'|^2 }$ is the Euclidean distance,
for some constants \(R\), \(L_1\), \(L_2 \geqslant 0\).
If additionally \(19 \max (1, \gamma) L_2 \leqslant \min(1, k)\) and the law of $(X_t,V_t)$ solving  \eqref{eq:kinetic} converges weakly for all initial condition as $t\rightarrow \infty$ to  a unique invariant measure
$\mu$ on $\R^{2d}$, then $\mu$
 satisfies an LSI with a constant
$C_\textnormal{LS} = C_\textnormal{LS}(C_0, K, L_1, L_2, R, \gamma;
L^\varphi)$.
\end{thm}

The proof of this result is given in Section~\ref{sec:Lipschitz_kinetic_proof}.

\begin{rem}
The assumptions of the kinetic perturbation theorem seem to be more restrictive
than the elliptic one.
First, the drift in the kinetic case must be
the sum of a positive linear transform
plus a perturbation term whose oscillation ``grows slowly enough''
compared to the linear term,
while in the elliptic case it only needs to satisfy a weak convexity condition.
This is because our proof is based on $W_1$-contraction of diffusion processes
and in the kinetic case such contraction is harder to establish
(e.g. compare Theorem~\ref{thm:kinetic_rc} to \cite{EberleReflectionCoupling}).
Second, the function $\varphi$ in the kinetic case must be Lipschitz
while in the elliptic case
it can be the sum of a Lipschitz and a bounded function,
due to the fact that the coupling of Theorem~\ref{thm:kinetic_rc}
does not allow us to obtain total variation bounds (dual to bounded functions)
as is done for the elliptic case.
\end{rem}

\subsection{Defective LSI approach}

A   measure $\mu\in\mathcal P(\R^d)$ is said to satisfy a defective log-Sobolev inequality if for all $f\geqslant 0$ with $\int f \dd \mu=1$,
\begin{equation}
\label{eq:defectiveLSI}
\int_{\R^d} f \ln f \dd \mu \leqslant A \int_{\R^d} \frac{|\na  f|^2 }{ f}\dd \mu  + B\,,
\end{equation}
for some constants $A$, $B \geqslant 0$. From \cite[Proposition 5.1.3]{BGLMarkov}, such a defective LSI, together with a Poincaré inequality
\begin{equation}
\label{eq:Poincare}
\forall f\in L^2(\mu),\qquad \int_{\R^d} \biggl(f - \int_{\R^d} f\dd \mu\biggr)^2 \dd \mu \ \leqslant \ C \int_{\R^d}|\na f|^2 \dd \mu
\end{equation}
for some constant $C>0$, implies an LSI for $\mu$ with constant  $A' = A + C(B+2)/4 $ (i.e. \eqref{eq:defectiveLSI} but with $A$ replaced by $A'$ and $B$ replaced by $0$).

In some non-reversible elliptic cases, a Poincaré inequality has been established in \cite{MonmarcheHighTemperature} (see Proposition~\ref{prop:Poincare} below). To improve this result into an LSI,  it is thus enough to obtain a defective LSI.

In the following for $\alpha$, $\beta\geqslant 1$ we write $\|f\|_\alpha = (\int |f|^\alpha\dd \mu)^{1/\alpha}$ and
\[ \|P_t \|_{\alpha\rightarrow \beta} = \sup\{ \|P_t f\|_\beta\ :\ f\in L^\alpha(\mu), \|f\|_\alpha=1\}\,,\]
where $\mu$ is the invariant measure of the semi-group $P_t$ considered.
The semi-group is said to be hypercontractive if there exist $t_0>0$, $\alpha < \beta$ such that $\|P_{t_0} \|_{\alpha\rightarrow \beta} <\infty$. In that case without loss of generality we can assume that $\alpha=1$, as the following easily follows from Hölder's inequality (the proof is given in Section~\ref{sec:hypoercon->LSI} for completeness):
\begin{lem}\label{lem:contractivity}
For all $\alpha$, $\gamma>1$,
\[\|P_t \|_{1\rightarrow \alpha} \leqslant \|P_t\|_{\alpha \rightarrow (\gamma  \alpha-1)/(\gamma-1)}^{\gamma\alpha -1}\,. \]
\end{lem}
In the reversible settings, it is well-known that hypercontractivity implies a defective LSI, see \cite{BGLMarkov}.  We show that it is also true  in the non-reversible case. For simplicity, we only consider the case where the diffusion matrix $\Sigma$ is constant, since this is anyway the case in \cite{MonmarcheHighTemperature}. In the reversible case, the proof relies on \cite[Proposition 5.2.6]{BGLMarkov}, whose proof requires reversibility. In the non-reversible case, we replace this result by the following (proven in Section~\ref{sec:defective_proof}):

\begin{prop}\label{prop:deffectiveLSI_main_section}
Let $(P_t)_{t\geqslant 0}$ be a diffusion semi-group with invariant measure $\mu$ and generator $b\cdot \na + \Sigma :\na^2$ where $\Sigma$ is a constant diffusion matrix and $b$ satisfies the one-sided Lipschitz condition
\begin{equation}\label{eq:onesidedLip}
\forall x,y\in\R^d,\qquad \bigl(b(x)- b(y)\bigr) \cdot (x-y) \leqslant
L|x-y|^2  \,.
\end{equation}
Then, for all $f\geqslant 0$ with $\int_{\R^d} f\dd \mu= 1$,
all $\alpha>0$ and all $t\geqslant 0$,
\[
\int_{\R^d} f \ln f \dd \mu \leqslant \frac{\alpha+1}\alpha \ln \|P_t\|_{1\rightarrow 1+\alpha}  + |\Sigma| \frac{e^{2Lt}-1}{2L} \int_{\R^d} \frac{|\na  f|^2 }{ f}\dd \mu \,.\]
\end{prop}

In the remaining of this section, we focus on the following elliptic case.

\begin{assu}\label{assu:contraction}
The semi-group $(P_t)_{t\geqslant 0}$,
whose generator reads $b\cdot \na + \sigma \Delta $ for $\sigma>0$,
admits an invariant measure $\mu$
and there exist $L$, $\rho$, $R>0$ such that
\begin{equation}\label{eq:condition_contraction_w2}
\forall x,y\in\R^d,\qquad \bigl(b(x)- b(y)\bigr) \cdot (x-y)
\leqslant \begin{cases}
- \rho  |x-y|^2 & \text{if $|x|\geqslant R$},\\
L|x-y|^2 & \text{otherwise}.
\end{cases}
\end{equation}
\end{assu}

Note that this assumption is different from \eqref{eq:condition_contraction},
which we imposed for the perturbation result in the elliptic case.
Under this assumption, hypercontractivity follows from the Harnack inequality established by Wang in \cite{WangLSIRicHess} (originally stated in the reversible case but the proof, recalled in Section~\ref{sec:defective}, is unchanged in the non-reversible one). More specifically, we get the following.

\begin{prop}\label{contractivity}
Let $\beta>\alpha>1$. Under Assumption~\ref{assu:contraction}, set
\[t_0  = \frac{2\beta   }{\sigma^2\rho(\alpha-1)}\,.\]
Then, for all $t>t_0$,
\[\| P_t\|_{\alpha\rightarrow \beta} \leqslant
\bigl(1+ 4  d + 2(L+ \rho)R^2\bigr)
\exp \Biggl(\frac{\beta  LR t }{2\sigma^2(\alpha-1)}
+ \frac18 \max\biggl(\frac{1 + 4   d}{t/t_0-1 }, 2\rho R^2\biggr)\Biggr)\,.\]
\end{prop}

Combining Propositions~\ref{prop:deffectiveLSI_main_section} and \ref{contractivity}  gives a defective LSI (see Corollary~\ref{cor:hypercontractivity}).
As a conclusion, we recall the following result from \cite[Theorems 1 and 2]{MonmarcheHighTemperature}.

\begin{prop}\label{prop:Poincare}
Under Assumption~\ref{assu:contraction},
assume furthemore that
\begin{equation}
\label{eq:sigma0}
\sigma \geqslant  \sigma_0\coloneqq(2L+\rho)
\frac{ (2L+\rho/2)R_*^2 + 2\sup\{-x\cdot b(x),\ |x|\leqslant R_*\}}{\rho d}\,,
\end{equation}
where $R_*=R(2+2L/\rho)^{1/d}$. Then $\mu$ satisfies the Poincaré inequality \eqref{eq:Poincare}  with constant
\begin{equation}\label{eq:C}
C = \frac{4\sigma }{\rho}\biggl(1+ \frac{\alpha (2L+\rho)R_*^2}{4d\sigma}\biggr)\,.
\end{equation}
\end{prop}

Thanks to  \cite[Proposition 5.1.3]{BGLMarkov}, the defective LSI of Corollary~\ref{cor:hypercontractivity}
and the Poincaré inequality of Proposition~\ref{prop:Poincare} yields the following.

\begin{cor}
\label{cor:LSIhigh_diffusitivity}
Under Assumption~\ref{assu:contraction}, provided furthemore \eqref{eq:sigma0}, $\mu$ satisfies an LSI with constant $C_\textnormal{LS}= A + C(B+2)/4$
where $A$, $B$, $C$ are respectively given
in \eqref{eq:A}, \eqref{eq:B}, \eqref{eq:C}.
\end{cor}

As in Section~\ref{sec:Lipschitz_elliptic}, in the present elliptic case, the LSI is equivalent to the entropy decay \eqref{eq:entropy_decay}.

\section{Proofs}\label{sec:proofs_Lipschitz}

\subsection{The elliptic case}\label{sec:Lipschitz_elliptic_proof}

Before proving the theorem, let us first show a key lemma
on the value function of stochastic optimal control problems.

\begin{lem}
\label{lem:elliptic-oc}
Let $U \subset \mathbb R^d$.
Under the conditions of Theorem~\ref{thm:elliptic},
consider the stochastic optimal control problem,
\[
V(T, x) = \sup_{\vphantom{\alpha : \alpha_t \in U}\nu} 
\sup_{\alpha : \alpha_t \in U} \mathbb E \biggl[\int_0^T
\bigl(\varphi(X_t^{\alpha,x}) - |\alpha_t|^2\bigr)\dd t \biggr]\,,
\]
where $\nu = \bigl(\Omega, F, (\mathcal F_\cdot), \mathbb P, (B_\cdot)\bigr)$
stands for a filter probability space with the usual conditions
and an \((\mathcal F_\cdot)\)-Brownian motion,
$\alpha$ is an $\mathbb R^d$-valued progressively measurable process such that
$\int_0^T \Expect \bigl[|\alpha_t|^m\bigr] \dd t$
is finite for every \(m \in \N\), and $X^{\alpha,x}$ solves
\[
X_0^{\alpha,x} = x\,,\qquad
\dd X_t^{\alpha,x} = \bigl(\tilde b(X_t^{\alpha,x}) + 2 \alpha_t\bigr) \dd t
+ \sqrt{2} \dd B_t \,.
\]
Then there exists $C' > 0$,
depending only on $\rho$, $L$, $R$,
such that for every $x$, $y \in \R^d$, every $T > 0$ and every $t \in (0,T]$,
we have
\[
|V(T,x) - V(T,y)| \leqslant 2 M^\varphi t
+ C'\biggl(\frac{2M^\varphi}{t} + L^\varphi\biggr) |x - y|\,.
\]
\end{lem}

\begin{proof}
The method has been demonstrated in \cite{ConfortiCouplReflControlDiffusion}
and we give a proof for the sake of completeness.
Fix $\varepsilon > 0$ and $(T,x) \in (0,+\infty)\times \R^d$.
Take an $\varepsilon$-optimal control $(\nu^\varepsilon,\alpha^\varepsilon)$
such that
\[
V(T,x) \leqslant \mathbb E \biggl[\int_0^T
\bigl(\varphi(X_t) - |\alpha^\varepsilon_t|^2\bigr)\dd t \biggr]
+ \varepsilon\,,
\]
where we denote $X = X^{\alpha^\varepsilon, x}$.
Construct the process $Y$ solving
\[
Y_0 = y\,,\qquad \dd Y_t = \bigl(\tilde b(Y_t) + 2\alpha^\varepsilon_t\bigr)
\dd t + \sqrt 2 (1 - 2e_t e_t^\mathrm T)\dd B_t\,,
\]
until $\tau \coloneqq \inf \{ t : X_t = Y_t \}$ and $Y_t = X_t$ henceforth,
where $e_t$ is defined by
\[
e_t = \begin{cases}
\frac{X_t - Y_t}{|X_t - Y_t|} & \text{if $X_t \neq Y_t$}, \\
(1,0,\ldots,0)^\mathrm{T} & \text{otherwise}.
\end{cases}
\]
Then the difference process $\delta X_t \coloneqq X_t - Y_t$ solves
\[
\dd \delta X_t = \bigl(\tilde b(X_t) - \tilde b(Y_t)\bigr) \dd t
+ 2 \sqrt 2 e_t e_t^\mathrm{T} \dd B_t\,.
\]
Thanks to the weak convexity condition \eqref{eq:condition_contraction},
there exist $C \geqslant 0$ and $\kappa > 0$ such that
\begin{align*}
\Expect \bigl[ |X_t - Y_t| \bigr] &\leqslant C e^{-\kappa t} |x - y|\,, \\
\mathbb P \bigl[ X_t \neq Y_t \bigr]
&\leqslant \frac{C e^{-\kappa t}}{t} |x - y|\,,
\end{align*}
where the first inequality is due to Eberle \cite{EberleReflectionCoupling}
and the second to the sticky coupling
\cite[Theorem 3]{EberleZimmerSticky}.
By the definition of $V$, we have
\[
V(T,y) \geqslant \mathbb E \biggl[\int_0^T
\bigl(\varphi(Y_t) - |\alpha^\varepsilon_t|^2\bigr)\dd t \biggr]\,.
\]
Hence by subtracting the expressions for $V(T,x)$ and $V(T,y)$,
we obtain
\begin{align*}
V(T,x) - V(T,y) &\leqslant \biggl(\int_0^t+\int_t^T\biggr)
\Expect\bigl[\varphi_1(X_s) - \varphi_1(Y_s)\bigr]\dd s
+ \int_0^T \Expect\bigl[\varphi_2(X_s) - \varphi_2(Y_s)\bigr]\dd s \\
&\leqslant 2M^\varphi t
+ \frac{C}{\kappa}\biggl(\frac{2M^\varphi}{t} + L^\varphi\biggr) |x - y|
+ \varepsilon\,.
\end{align*}
Taking $\varepsilon \to 0$ gives
the desired upper bound for $V(T,x)-V(T,y)$.
The lower bound follows by exchanging $x$ and $y$.
\end{proof}

Now we present the proof for the perturbation result in the elliptic case.
We will use the notion of viscosity solution
and we refer readers to \cite[Section 8]{CILUserGuide} for its definition.

\begin{proof}[Proof of Theorem~\ref{thm:elliptic}]
Under the conditions of  Theorem~\ref{thm:elliptic},
consider $m_t = \Law(Z_t)$
where $Z$ solves \eqref{eq:EDS} with initial distribution $m_0 = \mu_0$.
Then $m_t$, $\mu_0$ solve respectively
\begin{align*}
\partial_t m_t &= -\na\cdot (b m_t) + \Delta m_t\,, \\
0 = \partial_t \mu_0 &= -\na\cdot (b_0 \mu_0) + \Delta \mu_0\,
\end{align*}
where the first equation holds in the sense of distributions a priori.
By approximation arguments, we can show that $(t,x) \mapsto m_t(x)$
is continuous and a viscosity solution to the first equation.
Define the relative density $h_t = m_t / \mu_0$.
Then, it is a viscosity solution to
\begin{equation}
\label{eq:elliptic-h}
\partial_t h_t = \Delta h_t + \tilde b_t \cdot \nabla h_t + \varphi h_t\,,
\end{equation}
where $\varphi = -\na \cdot b_1 + b_1 \cdot \na \ln \mu_0$
and $\tilde b = -b + 2 \na \ln \mu_0$.
Notice that the value of $h_t$ can be given by the Feynman--Kac formula
\[
h_t(x) = \Expect \biggl[ \exp\biggl( \int_0^t \varphi(X^{t,x}_s) \dd s\biggr)
h_0(X^{t,x}_t)\biggr]\,,
\]
where $X^{t,x}$ solves
\[
X^{t,x}_0 = x,\qquad dX^{t,x}_s = \tilde b_{t-s} (X^{t,x}_s) \dd s + \sqrt 2 dB_t
\quad\text{for $s \in [0,t]$}.
\]

Suppose additionally that $\varphi$ is bounded and Lipschitz continuous.
Then applying synchronous coupling to the Feynman--Kac formula above,
we obtain a constant $M > 0$ such that
\[
M^{-1} \leqslant h(t,x) \leqslant M\qquad\text{and}\qquad
|h(t,x) - h(s,y)| \leqslant M \bigl(|t - s|^{1/2} + |x - y|\bigr)
\]
for every $t$, $s \in [0,T]$ and every $x$, $y \in \mathbb R^d$.
Taking the logarithm $u_t \coloneqq \ln h_t$
and using the fact that $h \mapsto \ln h$ is a strictly increasing
and $C^2$ mapping,
we obtain that $u_t$ is a bounded and uniformly continuous viscosity solution
to the HJB equation,
\begin{equation} \label{eq:elliptic-hjb}
\partial_t u_t = \Delta u_t + |\nabla u_t|^2 + \tilde b_t\cdot\nabla u_t
+ \varphi\,.
\end{equation}
The rest of the proof then amounts to linking the HJB equation
to the stochastic optimal control problem considered
in Lemma~\ref{lem:elliptic-oc}.

For $N \in \N$, consider the approximative HJB equation,
\begin{equation} \label{eq:elliptic-hjb-N}
\partial_t u^N_t = \Delta u^N_t + \sup_{\alpha:\lvert\alpha\rvert\leqslant N}
\{ 2\alpha\cdot\nabla u^N_t - |\alpha|^2\}
+ \tilde b\cdot \nabla u^N + \varphi\,,
\end{equation}
and the associated control problem,
\begin{equation} \label{eq:elliptic-oc-N}
V^N(T, x) = \sup_{\vphantom{\alpha:|\alpha_t| \leqslant N}\nu}
\sup_{\alpha:|\alpha_t| \leqslant N} \mathbb E \biggl[\int_0^T
\bigl(\varphi(X_t^{\alpha,x}) - |\alpha_t|^2\bigr)\dd t \biggr]\,,
\end{equation}
where $\nu$, $\alpha$ satisfy the conditions in the statement
of Lemma~\ref{lem:elliptic-oc}.
By Theorem IV.7.1 and the results in Sections V.3 and V.9
of \cite{FlemingSoner}, the value function $V^N$ defined
by \eqref{eq:elliptic-oc-N} is a bounded and uniformly continuous viscosity solution
to \eqref{eq:elliptic-hjb-N}.
Applying Lemma~\ref{lem:elliptic-oc} to the approximative problem
\eqref{eq:elliptic-oc-N},
we obtain a constant $C' > 0$ such that
\begin{align*}
\lvert V^N(t,x) - V^N(t,y)\rvert
&\leqslant C' \lVert\varphi\rVert_\textnormal{Lip}|x-y|\,,\\
\label{eq:elliptic-bound+lip}
|V^N(T,x) - V^N(T,y)| &\leqslant 2 M^\varphi t
+ C'\biggl(\frac{2M^\varphi}{t} + L^\varphi\biggr) |x - y|
\end{align*}
for every $t\in(0,T]$ and every $x$, $y \in\R^d$.
Hence if $w \in C^{1,2}\bigl([0,T) \times \R^d\bigr)$ is such that
$V^N - w$ attains a local maximum or a local minimum at $(t,x) \in [0,T)\times\R^d$,
then $\lvert\nabla w(t,x)\rvert
\leqslant C'\lVert\varphi\rVert_\textnormal{Lip}$ by the first inequality above.
This implies that $V^N$ is actually a viscosity solution
to the original \eqref{eq:elliptic-hjb} for
$N \geqslant C'\lVert\varphi\rVert_\textnormal{Lip}$.
Since both $u$ and $V^N$ are bounded and uniformly continuous on $[0,T]\times\R^d$,
we can apply the parabolic comparison for viscosity solutions on the whole space
\cite[Theorem 1]{DFOParabolicComparison}
to obtain $V^N(T,x) = u_T(x)$ for $N$ sufficiently large.
Therefore, for every $T > 0$, every $t \in (0,T]$ and every $x$, $y \in\R^d$,
we have
\begin{equation}
\label{eq:elliptic-bound+lip}
|u_T(x) - u_T(y)| \leqslant 2 M^\varphi t
+ C'\biggl(\frac{2M^\varphi}{t} + L^\varphi\biggr) |x - y|\,.
\end{equation}

Now we remove the additional assumption on $\varphi$
and take a sequence of $\varphi^n = \varphi^{n,1} + \varphi^{n,2}$
such that each of $\varphi^n$ is bounded and Lipschitz continuous,
$\lVert\varphi^{n,1}\rVert_\infty \leqslant M^\varphi$,
$\lVert\varphi^{n,2}\rVert_\textnormal{Lip} \leqslant L^\varphi$
for all $n \in \N$,
and $\varphi^n \to \varphi$ locally uniformly.
For each $n$, consider the equation
\[
\partial_t h^n_t = \Delta h^n_t + \tilde b \cdot \nabla h^n_t
+ \varphi^n h^n_t
\]
and let $h^n$ be the solution given by the Feynman--Kac formula
with the initial condition $h^n_0 = 1$.
Taking the limit $n \to +\infty$ in the Feynman--Kac formulas
and using the dominated convergence theorem,
we obtain that $h^n_T \to h_T$ pointwise.
Yet, each $u^n_T \coloneqq \ln h^n_T$ satisfies
the bound \eqref{eq:elliptic-bound+lip} when $u$ is replaced by $u^n$.
So taking the limit, we obtain that \eqref{eq:elliptic-bound+lip}
still holds without the additional assumption on $\varphi$.

Denote the Gaussian kernel in $\mathbb R^d$ by
$g^\varepsilon = (2\pi\varepsilon)^{-d/2} \exp(-\lvert x\rvert^2\!/2\varepsilon)$.
We decompose $u_T$ in the following way:
\[
u_T = u_T \star g^\varepsilon + (u_T - u_T \star g^\varepsilon).
\]
Thanks to \eqref{eq:elliptic-bound+lip}, we find that
the first term is uniformly Lipschitz, and the second term is uniformly bounded.
Then we apply successively the Holley--Stroock
and Aida--Shigekawa perturbation lemmas \cite{HolleyStroockLSI,AidaShigekawaLSI},
and obtain that the flow of measures
\[(m_T)_{T \geqslant 1}
= (\mu_0 \exp u_T)_{T \geqslant 1}\]
satisfies a uniform log-Sobolev inequality
(see \cite[Theorem 2.7]{CattiauxGuillinFunctional} for an explicit constant
for Aida--Shigekawa).
Noticing that the LSI is stable under the weak convergence of measures,
we take the limit \(T \to +\infty\) and conclude.
\end{proof}

\begin{rem}
We exploit the properties of viscosity solution to the HJB equation
\eqref{eq:elliptic-hjb} instead of classical solution,
contrary to what is done by Conforti \cite{ConfortiCouplReflControlDiffusion}.
The main reason for this is that we wish to be able to treat the kinetic,
therefore degenerate elliptic, case in the same framework,
for which the existence of classical solution,
despite the system's hypoellipticity,
is lacking in classical literatures of stochastic optimal control
to our knowledge.
\end{rem}

\subsection{The kinetic case}\label{sec:Lipschitz_kinetic_proof}

As in the previous section, we first establish a lemma
on the kinetic stochastic optimal control problem.

\begin{lem}
\label{lem:kinetic-oc}
Let $U \subset \mathbb R^d$.
Under the conditions of Theorem~\ref{thm:kinetic},
consider the stochastic optimal control problem,
\[
V(T, z) = \sup_{\vphantom{\alpha : \alpha_t \in U}\nu}
\sup_{\alpha : \alpha_t \in U} \mathbb E \biggl[\int_0^T
\bigl(\varphi(Z_t^{\alpha,z}) - \gamma |\alpha_t|^2\bigr)\dd t \biggr]\,,
\]
where $\nu = \bigl(\Omega, F, (\mathcal F_\cdot), \mathbb P, (B_\cdot)\bigr)$
stands for a filter probability space with the usual conditions
and an \((\mathcal F_\cdot)\)-Brownian motion,
$\alpha$ is an $\mathbb R^d$-valued progressively measurable process such that
$\int_0^T \Expect \bigl[|\alpha_t|^m\bigr] \dd t$
is finite for every \(m \in \N\),
and $Z^{\alpha,z} = (X^{\alpha,z},Y^{\alpha,z})$ solves
\[
Z_0^{\alpha,x} = z\,,\quad
\Bigg\{
\begin{aligned}
\dd X_t^{\alpha,z} &= V_t^{\alpha,z}\dd t\,,\\
\dd V_t^{\alpha,z} &= \bigl( - \gamma V_t^{\alpha,z} - \nabla U(X_t^{\alpha,z})
+ G(X_t^{\alpha,z}, -V_t^{\alpha,z}) + 2\gamma\alpha_t \bigr) \dd t
+ \sqrt{2 \gamma} \dd B_t\,.
\end{aligned}
\]
Then there exists $C' > 0$,
depending only on $K$, $L_1$, $L_2$, $R$, $\gamma$,
such that for every $z$, $z' \in \R^d$,
we have
\[
|V(T,z) - V(T,z')| \leqslant
C'L^\varphi|z - z'|\,.
\]
\end{lem}

\begin{proof}
Fix $\varepsilon > 0$ and $(T,z) \in (0,+\infty)\times \R^{2d}$.
Take an $\varepsilon$-optimal control $(\nu^\varepsilon,\alpha^\varepsilon)$
such that
\[
V(T,z) \leqslant \mathbb E \biggl[\int_0^T
\bigl(\varphi(Z_t) - \gamma |\alpha^\varepsilon_t|^2\bigr)\dd t \biggr]
+ \varepsilon\,,
\]
where we denote $(X,V) = Z = Z^{\alpha^\varepsilon, z}$.
Using $\gamma t$ as the new time variable
and $\gamma^{-1}X$ as the new space variable,
and noticing that $- \nabla U(x) + G(x,-v) = - Kx + g(x,v)$,
we can apply Theorem~\ref{thm:kinetic_rc} in the appendix
to construct the processes $Z'_n = (X'_n, Y'_n)$ solving
\[
Z'_{n,0} = z'\,,\qquad
\Bigg\{
\begin{aligned}
\dd X'_{n,t} &= V'_{n.t}\dd t\,,\\
\dd V'_{n,t} &= \bigl( - \gamma V'_{n,t} - KX'_{n,t}
+ g(Z'_{n,t}) + 2\gamma\alpha_t \bigr) \dd t + \sqrt{2 \gamma} \dd B'_{n,t}\,,
\end{aligned}
\]
where $B'_n$ are Brownian motions,
and there exist constant $C_1 \geqslant 1$, $\kappa > 0$ such that
\[
\limsup_{n\to+\infty} \Expect\bigl[ |Z_t - Z'_{n,t}| \bigr]
\leqslant C_1 e^{-\kappa t} \rho(z,z')\quad\text{for $t \geqslant 0$}.
\]
By the definition of $V$, we have
\[
V(T,z') \geqslant \mathbb E \biggl[\int_0^T
\bigl(\varphi(Z'_{n,t}) - \gamma|\alpha^\varepsilon_t|^2\bigr)\dd t \biggr]\,.
\]
Hence by subtracting the expressions for $V(T,z)$ and $V(T,z')$,
we obtain
\[
V(T,z) - V(T,z') \leqslant
\int_0^T \Expect\bigl[\varphi(Z_s) - \varphi(Z'_{n,s})\bigr]\dd s + \varepsilon
\leqslant L^\varphi\int_0^T\Expect\bigl[|Z_s - Z'_{n,s}|\bigr]\dd s
+ \varepsilon\,,
\]
By Fatou's lemma we have
\begin{multline*}
\liminf_{n\to+\infty}\int_0^T\Expect\bigl[|Z_s - Z'_{n,s}|\bigr]\dd s
\leqslant \int_0^T \liminf_{n\to+\infty}\Expect\bigl[|Z_s - Z'_{n,s}|\bigr]\dd s
\leqslant \int_0^T C_1e^{-\kappa t}\rho(z,z')\dd s \\
< \frac{C_1}{\kappa}\rho(z,z')\,.
\end{multline*}
So taking $\varepsilon \to 0$ and exchanging $x$ and $y$, we obtain
\[
|V(T,z) - V(T,z')| \leqslant \frac{C_1L^\varphi}{\kappa} \rho(z,z')\,.
\]
Setting the interpolation \(z_s = (1 - s) z + s z'\),
by the previous inequality we have
\begin{align*}
|V(T,z) - V(T,z')|
&= \sum_{i=0}^{N-1} \bigl|V\bigl(T, z_{(i+1)/N}\bigr)
- V\bigl(T,z_{i/N}\bigr)\bigr| \\
&\leqslant \frac{C_1L^\varphi}{\kappa}
\sum_{i=0}^{N-1} \rho \bigl(z_{(i+1)/N}, z_{i/N}\bigr) \\
&= \frac{C_1L^\varphi|z - z'|}{\kappa}
\frac 1N\sum_{i=0}^{N-1} \frac{\rho \bigl(z_{(i+1)/N}, z_{i/N}\bigr)}
{\bigl|z_{(i+1)/N} - z_{i/N}\bigr|}\,.
\end{align*}
As the uniform convergence
\(\limsup_{z \to z'}\rho(z, z')/|z - z'| \leqslant C_2\)
holds,
we have
\[
\limsup_{N\to +\infty}
\frac 1N\sum_{i=0}^{N-1} \frac{\rho \bigl(z_{(i+1)/N}, z_{i/N}\bigr)}
{\bigl|z_{(i+1)/N} - z_{i/N}\bigr|} \leqslant C_2\,.
\]
Thus taking the limit
\(N \to +\infty\) in the inequality above concludes the proof.
\end{proof}

\begin{proof}[Proof of Theorem~\ref{thm:kinetic}]
Under the conditions of  Theorem~\ref{thm:kinetic},
consider $m_t = \Law(Z_t)$
where $Z$ solves \eqref{eq:kinetic} with the initial distribution
\[
m_0(\dd x\dd v) = \mu_0(\dd x\dd v)
\propto \exp \biggl( - U(x) - \frac 12 |v|^2 \biggr) \dd x\dd v\,,
\]
which is the unique invariant measure of the diffusion \eqref{eq:kinetic}
when $G=0$.
By the tensorization property, $\mu_0$ satisfies an LSI with constant
$\max (1, C_0)$.
The measures $m_t$, $\mu_0$ solve respectively
\begin{align*}
\partial_t m_t &=
\gamma \Delta_v m_t
+ \nabla_v \cdot \bigl[ m_t \bigl(\gamma v + \nabla U - G(x,v)\bigr) \bigr]
- v \cdot \nabla_x m_t\,, \\
0 = \partial_t \mu_0 &=
\gamma \Delta_v \mu_0
+ \nabla_v \cdot [ \mu_0 (\gamma v + \nabla U) ]
- v \cdot \nabla_x \mu_0\,,
\end{align*}
where the first equation holds in the sense of viscosity.
Define the relative density $h_t = m_t / \mu_0$.
Then, it is a viscosity solution to
\begin{equation}
\label{eq:kinetic-h}
\partial_t h_t = \gamma \Delta_v h_t
+ \bigl( -\gamma v + \nabla U(x) - G(x,v) \bigr) \cdot \nabla_v h
- v \cdot \nabla_x h + \varphi h\,,
\end{equation}
where $\varphi = - \nabla_v G(x,v) + G(x,v) \cdot v$.
Taking the logarithm $u_t \coloneqq \ln h_t$
and using the fact that $h \mapsto \ln h$ is a strictly increasing
and $C^2$ mapping,
we obtain that $u_t$ is a viscosity solution
to the kinetic HJB equation,
\begin{equation} \label{eq:kinetic-hjb}
\partial_t u_t = \gamma \Delta_v u_t + \gamma |\nabla_v u_t|^2
+ \bigl( -\gamma v + \nabla U(x) - G(x,v) \bigr) \cdot \nabla_v u
- v \cdot \nabla_x u + \varphi\,.
\end{equation}

Now, on the formal level the kinetic HJB equation \eqref{eq:kinetic-hjb}
is related to the optimal control problem considered
in Lemma~\ref{lem:kinetic-oc}:
if the domain of control in the lemma is unrestricted, i.e. $U = \mathbb R^d$,
then we expect to have
\[
u_T(x,v) = V(T,x,-v)\,.
\]
We then argue as in the proof of Theorem~\ref{thm:elliptic}
(suppose $\varphi$ is regular enough, then restrict the domain of control,
finally approximate for general $\varphi$) to validate this claim.
Then by Lemma~\ref{lem:kinetic-oc}, for every $z$, $z' \in \R^{2d}$, we have
\[
|u_T(z) - u_T(z')| \leqslant C' L^\varphi| z-z'|\,.
\]
We conclude as in the end of the proof of Theorem~\ref{thm:elliptic}.
\end{proof}

\section{Defective log-Sobolev inequality}\label{sec:defective_proof}

\subsection{From Hypercontractivity to defective LSI}\label{sec:hypoercon->LSI}

\begin{proof}[Proof of Proposition~\ref{prop:deffectiveLSI_main_section}]
From \cite[Theorem 1]{MonmarcheAlmostSure}, the one-sided Lipschitz condition \eqref{eq:onesidedLip} implies that $|\na P_t f|\leqslant e^{tL}P_t|\na f|$ for all $t\geqslant 0$ for all $f$. Then, classically, for $f\geqslant 0$ with $\int  f \dd \mu=1$,
\begin{align*}
\int_{\R^d} P_t f \ln P_t f \dd \mu & =
\int_{\R^d} f \ln f \dd \mu - \int_0^t \int_{\R^d} \frac{|\Sigma^{1/2} \na P_s f|^2}{P_s f}\dd \mu \dd s\\
& \geqslant  \int_{\R^d} f \ln f \dd \mu - |\Sigma| \int_0^t \int_{\R^d} \frac{e^{2sL}(P_s|\na  f|)^2 }{P_s f}\dd \mu \dd s \\
& \geqslant \int_{\R^d} f \ln f \dd \mu - |\Sigma| \int_0^t  \int_{\R^d} \frac{e^{2sL}|\na  f|^2 }{ f}\dd \mu\dd s\,,
\end{align*}
where $\bigl(P_s(|\na f|)\bigr)^2 \leqslant P_s(|\na f|^2/f) P_s(f)$
(by Cauchy--Schwarz) and the invariance of $\mu$ by $P_s$ were used in the last inequality. In other words,
\begin{align*}
\int_{\R^d} f \ln f \dd \mu & \leqslant \int_{\R^d} P_t f \ln P_t f \dd \mu + |\Sigma| \frac{e^{2Lt}-1}{2L} \int_{\R^d} \frac{|\na  f|^2 }{ f}\dd \mu \\
&= \frac1\alpha \int_{\R^d} P_t f \ln (P_t f)^\alpha \dd \mu + |\Sigma| \frac{e^{2Lt}-1}{2L} \int_{\R^d} \frac{|\na  f|^2 }{ f}\dd \mu \,,
\end{align*}
for every $\alpha>0$. By Jensen's inequality applied to the probability measure $P_t f \mu$,
\[\int_{\R^d} P_t f \ln (P_t f)^\alpha \dd \mu \leqslant \ln \int_{\R^d} (P_t f)^{1+\alpha} \dd \mu \leqslant (1+\alpha) \ln \|P_t\|_{1\rightarrow 1+\alpha}\,, \]
since $\|P_t f\|_1=1$.
\end{proof}

\begin{proof}[Proof of Lemma~\ref{lem:contractivity}]
Using Hölder's inequality, for $f\geqslant 0$ with $\int f \dd \mu =1$ (so that $\|P_t f\|_1=1$ since $\mu$ is invariant by $P_t$),
\begin{align*}
\|P_t f\|_\alpha^\alpha  & \leqslant \|P_t f\|_1^{1/\gamma} \|P_t f\|_{(\gamma  \alpha-1)/(\gamma-1)}^{(\gamma\alpha-1)/\gamma} \\
& \leqslant \|P_t \|_{\alpha\rightarrow (\gamma  \alpha-1)/(\gamma-1)}^{(\gamma\alpha-1)/\gamma} \|P_t f\|_{\alpha}^{(\gamma\alpha-1)/\gamma}.
\end{align*}
Dividing by $\|P_t f\|_{\alpha}^{(\gamma\alpha-1)/\gamma}$ concludes.
\end{proof}

\subsection{Hypercontractivity in the elliptic case}\label{sec:defective}

Next, we recall  (here in a non-reversible settings -- which doesn't change the proof -- and only in the flat space ; also with explicit constants) the Harnack inequality of \cite{WangLSIRicHess}.

\begin{prop}\label{prop:Harnack}
Assume that there exists $K>0$ such that
\begin{equation}\label{assu:Wang}
\forall x,\,y\in\R^d,
\qquad (x-y) \cdot \bigl(b(x) - b(y)\bigr) \leqslant K|x-y| \,.
\end{equation}
Then, for all $t\geqslant 0$, all $x$, $y\in\R^d$ and all $\alpha>1$,
\[\bigl(P_t f(y)\bigr)^\alpha  \leqslant (P_t f^\alpha) (x) \exp \Biggl(\frac{\alpha }{2\sigma^2(\alpha-1)}\biggl(K^2t + \frac{|x-y|^2}{t}\biggr)\Biggr)\,. \]
\end{prop}

Notice that, in particular, Assumption~\ref{assu:contraction} implies \eqref{assu:Wang}  with $K= LR$.

\begin{proof}

For two initial conditions $x\neq y$ and a final time $T>0$, let $X$, $Y$ solve
\begin{alignat*}{2}
X_0 &= x\,, &\qquad\dd X_t &= b(X_t) \dd t + \sigma \dd B_t\,,\\
Y_0 &= y\,, &\dd Y_t &= b(Y_t) \dd t + \sigma \dd B_t + \xi e_t \dd t\,,
\end{alignat*}
where $e_t =(X_t-Y_t)/|X_t-Y_t|$ for $t< \tau \coloneqq \inf\{s\geqslant 0,\ X_s=Y_s\}$ and $e_t=0$ for $t\geqslant \tau$ (so that in particular $X_t=Y_t$ for $t\geqslant \tau$) and $\xi = K + |x-y|/T$.

Since the norm is $\mathcal C^2$ outside the origin, we can apply It\=o's formula up to time $\tau$ to get, for $t<\tau$,
\begin{align*}
\dd |X_t - Y_t| &= e_t\cdot \bigl(b(X_t)-b(Y_t)\bigr) \dd t -\xi_t \dd t   \\
& \leqslant - \frac{|x-y|}{T}\dd t\,.
\end{align*}
This implies that $\tau\leqslant T$, and thus $X_T=Y_T$. By Girsanov's theorem,
\[P_T f(y) = \mathbb E [f(Y_T)  R] \,,\quad \text{with}\quad
R= e^{\frac{\xi }{\sigma}\int_0^\tau e_t \cdot \dd B_t - \frac{\xi^2}{2\sigma^2}\tau  }\,,\]
so that, by Hölder's inequality, for $f\geqslant 0$ and $\alpha>1$, using that $Y_T=X_T$,
\[\bigl(P_T f(y)\bigr)^\alpha
\leqslant (P_T f^\alpha)(x)
\bigl(\Expect R^{\alpha/(\alpha-1)}\bigr)^{\alpha-1} \]
with
\[ \bigl(\Expect R^{\alpha/(\alpha-1)}\bigr)^{\alpha-1}
\leqslant \exp\biggl(\frac{\alpha }{4\sigma^2(\alpha-1)} \xi^2 T\biggr)
\leqslant \exp\Biggl(\frac{\alpha }{2\sigma^2(\alpha-1)}\biggl(K^2T + \frac{|x-y|^2}{T}\biggr)\Biggr)\,. \qedhere\]
\end{proof}

\begin{lem}\label{lem:Lyapunov}
Under Assumption~\ref{assu:contraction},
\[\int_{\R^d} e^{\delta |x-y|^2} \mu(\dd x)\mu(\dd y)
\leqslant \bigl(1+ 4  d + (2 L+ 8 \delta)R^2\bigr)\exp\Biggl(\delta \max\biggl(\frac{1 + 4   d}{2(\rho-4\delta) },R^2\biggr)\Biggr)\,,\]
for all $\delta \in (0,\rho/4)$.
\end{lem}
\begin{proof}
Let $\mathcal L_2$ be the generator of two independent diffusion processes
satisfying \eqref{eq:EDS}, so that $\mu\otimes \mu$ is invariant by $\mathcal L_2$.
For $\delta\in(0,\rho/4)$, consider the function $V(x,y) = e^{\delta|x-y|^2}$.
Then
\begin{align*}
\frac{\mathcal L_2 V(x,y)}{V(x,y)} & = 2\delta (x-y)\cdot\bigl(b(x)-b(y)\bigr)
+ 4 \delta d + 8 \delta^2 |x-y|^2 \\
& \leqslant \begin{cases}
4 \delta d + (8 \delta^2- 2\delta \rho) |x-y|^2 & \text{if }|x-y|\geqslant R\\
4 \delta d + (2\delta L+ 8 \delta^2) R^2     & \text{otherwise}
\end{cases}\\
& \leqslant  -\delta \1_{|x-y|\geqslant R_*} + \delta\bigl(4  d + (2 L+ 8 \delta)R^2\bigr)\1_{|x-y|<R_*}
\end{align*}
with
\[R_*^2 = \max\biggl(\frac{1 + 4   d}{2(\rho-4\delta) },R^2\biggr)\,. \]
Hence
\begin{align*}
\mathcal L_2V(x,y) & \leqslant - \delta  V(x,y) + \delta\bigl(1+ 4  d +(2 L+ 8 \delta) R^2\bigr) \1_{|x-y|\leqslant R_*} V(x,y) \\
& \leqslant - \delta  V(x,y) + \delta\bigl(1+ 4  d +(2 L+ 8 \delta) R^2\bigr) e^{\delta R_*^2}\,.
\end{align*}
Integrating with respect to $\mu\otimes \mu$, the left hand side vanishes and we get
\[\int_{\R^d} V(x,y)\mu(\dd x)\mu(\dd y) \leqslant
\bigl(1+ 4  d +(2 L+ 8 \delta) R^2\bigr) e^{\delta R_*^2}\,,\]
as announced.
\end{proof}

\begin{proof}[Proof of Proposition~\ref{contractivity}]
Let $f\geqslant 0$ be such that $\mu (f^\alpha)=1$. By Proposition~\ref{prop:Harnack} (with $K=LR$), for any $y\in\R^d$,
\begin{align*}
1 &= \int_{\R^d} P_t(f^\alpha)(x) \mu(\dd x)  \\
& \geqslant \bigl(P_t f(y)\bigr)^\alpha  \int_{\R^d} \exp\Biggl(\frac{- \alpha }{2\sigma^2(\alpha-1)}\biggl(K^2t + \frac{|x-y|^2}{t}\biggr)\Biggr)\mu(\dd x)\,.
\end{align*}
As a consequence, for $\beta>\alpha$,
\begin{align*}
\int_{\R^d} \bigl(P_t f(y)\bigr)^{\beta } \mu(\dd y)
& \leqslant \int_{\R^d}\Biggl[\int_{\R^d} \exp\Biggl(\frac{-\alpha  }{2\sigma^2(\alpha-1)}\biggl(Kt + \frac{|x-y|^2}{t}\biggr)\Biggr)\mu(\dd x)\Biggr]^{-\beta/\alpha} \mu(\dd y) \\
& \leqslant \int_{\R^d} \int_{\R^d} \exp\Biggl(\frac{\beta   }{2\sigma^2(\alpha-1)} \biggl(Kt + \frac{|x-y|^2}{t}\biggr)\Biggl)\mu(\dd x) \mu(\dd y)\,.
\end{align*}
Conclusion follows from \eqref{lem:Lyapunov} and using that $t/t_0 \leqslant 1$.
\end{proof}

To conclude, gathering Propositions~\ref{prop:deffectiveLSI_main_section} and  \ref{contractivity}, we get the following:
\begin{cor}\label{cor:hypercontractivity}
Assume \eqref{eq:condition_contraction_w2} for some $L$, $R\geqslant 0$ and $\rho>0$.
Then for all $f\geqslant 0$ with $\int_{\R^d} f \dd \mu =1$,
we have
\begin{equation}\label{eq:deffectiveLSI}
\int_{\R^d} f \ln f \leqslant   A \int_{\R^d} \frac{|\na  f|^2 }{ f}\dd \mu + B
\end{equation}
with
\begin{align}
A & = \frac{\sigma^2}{2L}\Biggl(\exp\biggl(\frac{24L}{\sigma^2\rho }\biggr)-1\Biggr)\label{eq:A}\,,\\
B &= 6 \ln\bigl(1+ 4  d + 2(L+ \rho)R^2\bigr)
+ \frac{108LR}{\sigma^4\rho} + \frac34 \max \bigl(1 + 4d , 2\rho R^2\bigr)\label{eq:B}
\end{align}
(taking for $A$ the limit as $L\rightarrow 0$ of this expression if $L=0$).
\end{cor}
\begin{proof}
For simplicity, take $\alpha=1$ in Proposition~\ref{prop:deffectiveLSI_main_section},  $\alpha=\gamma=2$ in Lemma~\ref{lem:contractivity} and $t=2t_0$ in  Proposition~\ref{contractivity}, we end up with
\begin{align*}
\int_{\R^d} f \ln f  & \leqslant 2 \ln \|P_t\|_{1\rightarrow 2 }  + \sigma^2 \frac{e^{2Lt}-1}{2L} \int_{\R^d} \frac{|\na  f|^2 }{ f}\dd \mu  \\
& \leqslant 6 \ln \|P_t\|_{2 \rightarrow 3 }  + \sigma^2 \frac{e^{2Lt}-1}{2L} \int_{\R^d} \frac{|\na  f|^2 }{ f}\dd \mu  \,,
\end{align*}
and conclusion follows  the expression given in Proposition~\ref{contractivity}.
\end{proof}

\appendix

\section{Reflection coupling for kinetic diffusions} \label{sec:couplage_cinetique}

\begin{thm}[Coupling by reflection for kinetic diffusions]
\label{thm:kinetic_rc}
Let \(\bigl(\Omega,\mathcal F, (\mathcal F_t)_{t\geqslant 0}, \mathbb P\bigr)\)
be a filtered probability space satisfying the usual conditions.
Let \(X\), \(V\) be \(\mathbb R^d\)-valued
continuous and adapted processes,
and \(\alpha\) be an \(\mathbb R^d\)-valued progressively measurable process
solving
\begin{equation}
\label{eq:kinetic-diffusion}
\begin{aligned}
\dd X_t &= V_t \dd t\,, \\
\dd V_t &= \alpha_t \dd t + \bigl( - V_t - KX_t + g(X_t, V_t) \bigr) \dd t
+ \sqrt 2 \dd B_t\,,
\end{aligned}
\end{equation}
for \(t \geqslant 0\),
where \(K\) is a \(d \times d\) symmetric and positive-definite matrix,
\(g : \mathbb R^{2d} \to \mathbb R^{d}\) is a Lipschitz continuous function,
and \((B_t)_{t\geqslant 0}\) is an \((\mathcal F_\cdot)\)-Brownian motion
in \(d\) dimensions.
Let \(X'_0\), \(V'_0\) be \(\mathbb R^d\)-valued
and \(\mathcal F_0\)-measurable random variables.
Denote by \(k\) the smallest eigenvalue of \(K\).
Suppose that \(\int_0^T \Expect \bigl[|\alpha_t|^2\bigr] \dd t\)
is finite for every \(T > 0\),
and \(X_0\), \(V_0\), \(X'_0\), \(V'_0\) are all square-integrable.
If there exist nonnegative constants \(R\), \(L_1\), \(L_2\)
such that for every \(z\), \(z' \in \mathbb R^{2d}\), we have
\[
\bigl| g(z) - g(z') \bigr| \leqslant \begin{cases}
L_1 |z - z'| & \text{if $|x - x'| + |v - v'| \leqslant R$}, \\
L_2 |z - z'| & \text{otherwise},
\end{cases}
\]
with \(L_2 < \frac 1{19} \min (1, k)\) and \(L_2 \leqslant L_1\),
then upon enlarging the probability space,
we can construct a sequence of
continuous and adapted processes \(X'_{n}\), \(V'_{n}\) such that
\begin{enumerate}
\item their initial values are given by \(X'_0\), \(V'_0\),
that is, \(X'_{n,0} = X'_0\) and \(V'_{n,0} = V'_0\);
\item they solve
\begin{equation}
\label{eq:kinetic-diffusion-prime}
\begin{aligned}
\dd X'_{n,t} &= V'_{n,t} \dd t\,, \\
\dd V'_{n,t} &= \alpha_t \dd t
+ \bigl( - V'_{n,t} - K X'_{n,t} + b(X'_{n,t}, V'_{n,t}) \bigr) \dd t
+ \sqrt 2 \dd B'_{n,t}\,,
\end{aligned}
\end{equation}
for \((\mathcal F_\cdot)\)-Brownians
\((B'_{n,t})^{\phantom{\prime}}_{t\geqslant 0}\);
\item and finally, there exists constants
$C_1$, $C_2\geqslant 1$, $\kappa > 0$
and a continuous function of quadratic growth
\(\rho : \mathbb R^{2d} \times \mathbb R^{2d} \to \mathbb R\),
all explicitly expressible by \(K\), \(R\), \(L_1\), \(L_2\), such that
\begin{equation}
\label{eq:kinetic-wasserstein-contraction}
\limsup_{n\to+\infty} \Expect \bigl [ |Z_t - Z'_{n,t}| \bigr]
\leqslant C_1 e^{-\kappa t} \Expect\bigl[ \rho (Z_0, Z'_0) \bigr]
\quad\text{for $t \geqslant 0$}, \\
\end{equation}
and uniformly in $z'$, we have
\[
\limsup_{z \to z'}\frac{\rho(z, z')}{|z - z'|} \leqslant C_2\,.
\]
\end{enumerate}
\end{thm}

\begin{rem}
We develop a translation-invariant version
of the additive metric constructed in \cite{EGZCoupling}
under which the difference processes,
\(\delta X = X - X'\), \(\delta V = V - V'\),
are contractive in average.
Our assumptions are an improvement over \cite[Theorem 2.16]{KRTY}
although we do not elaborate on mean field dependence.
Also, we can recover
the contraction in $W_1$ distance from the transport cost $\rho$
by a limiting procedure (as is done in the proof
of Lemma~\ref{lem:kinetic-oc}).
So our approach can be used to achieve the $W_1$-contraction
of \cite[Theorem 5]{SchuhGlobalContractivity} with simpler calculations.
\end{rem}

\begin{proof}

For \(x\), \(x'\), \(v\), \(v' \in \mathbb R^d\),
introduce the variables \(q = x + v\), \(q' = x' + v'\)
and denote \(\delta x = x - x'\), \(\delta v = v - v'\), \(\delta q = q - q'\).
Define
\[
r (z,z') = \theta |\delta x| + |\delta q| = \theta |x - x'| + |q - q'|
\]
where
\[
\theta \coloneqq 2\max \bigl( |K| + L_1, 1 \bigr).
\]
Denote \(r_0 = (\theta + 1)R\).

\proofstep{Reflection-synchronous coupling}
Fix an \(n \in \mathbb N\).
Let us construct the desired processes \(Z'_n = (X'_n, V'_n)\).
Find Lipschitz-continuous
\(\refc_n\), \(\sync_n : \mathbb R^{2d} \times \mathbb R^{2d} \to \mathbb R\)
satisfying \(\refc_n^2 + \sync_n^2 = 1\) and
\[
\refc_n (z, z') = \begin{cases}
0 & \text{if $r(z,z') \geqslant r_0 + n^{-1}$ or $|\delta q| \leqslant n^{-1}$},\\
1 & \text{if $r(z,z') \leqslant r_0$ and $|\delta q| \geqslant 2n^{-1}$}.
\end{cases}
\]
Upon enlarging the filtered probability space,
we can also find another \((\mathcal F_\cdot)\)-Brownian motion \(B''\)
that is independent from \(B\).
\footnote{The additional Brownian motion \(B''\) can be shared between
stochastic processes \(Z'_n\) with different index \(n\),
so we do not need to extend the probability space infinitely.}
Let \(Z'_n = (X'_n, V'_n)\) solve
\begin{align*}
\dd X'_{n,t} &= V'_{n,t} \dd t\,, \\
\dd V'_{n,t} &= \alpha_t \dd t
+ \bigl( - V'_{n,t} - K X'_{n,t} + g(X'_{n,t}, V'_{n,t}) \bigr) \dd t \\
&\quad
+ \refc_n(Z_{t}, Z'_{n,t}) (1 - 2e_{n,t} e_{n,t}^\mathrm{T}) \sqrt 2\dd B^\refc_{n,t}
+ \sync_n(Z_t, Z'_{n,t}) \sqrt 2\dd B^\sync_{n,t}\,,
\end{align*}
with initial value \(Z'_{n,0} = (X'_0, V'_0)\),
where \(e_{n,t}\) is defined by
\[
e_{n,t} = \begin{cases}
\frac{Q_t - Q'_{n,t}}{|Q_t - Q'_{n,t}|} & \text{if $Q_t \neq Q'_{n,t}$}, \\
(1,0,\ldots,0)^\mathrm{T} & \text{otherwise},
\end{cases}
\]
for \(Q_t \coloneqq X_t + V_t\) and \(Q'_{n,t} \coloneqq X'_{n,t} + V'_{n,t}\),
and \(B^\refc_n\), \(B^\sync_n\) are defined by
\begin{align*}
\dd B^\refc_{n,t} &= \refc_n(Z_t, Z'_{n,t}) \dd B_t
+ \sync_n(Z_t, Z'_{n,t}) \dd B''_t, \\
\dd B^\sync_{n,t} &= \sync_n(Z_t, Z'_{n,t}) \dd B_t
- \refc_n(Z_t, Z'_{n,t}) \dd B''_t.
\end{align*}
The solution \(Z'_n\) to this system of equations is well defined:
they have Lipschitz-continuous coefficients in \(Z'_n\),
or in \(X'_n\) and \(V'_n\)
(although the Lipschitz constants explode when \(n \to +\infty\)),
so the existence and uniqueness of (strong) solution follow from
Cauchy--Lipschitz arguments.
By Lévy's characterization, \(B^\refc_n\), \(B^\sync_n\)
are independent Brownian motions
and therefore if we define \(B'_n\) by
\[
\dd B'_{n,t} = \refc_n(Z_{t}, Z'_{n,t})
(1 - 2e_{n,t} e_{n,t}^\mathrm{T}) \dd B^\refc_{n,t}
+ \sync_n(Z_{t}, Z'_{n,t}) \dd B^\sync_{n,t}\,,
\]
then \(B'_n\) is also a Brownian motion.
Hence \(Z'_n\) satisfies indeed \eqref{eq:kinetic-diffusion-prime}
and it remains only to verify the last claim.

\proofstep{Difference process}
Denote \(\refc_{n,t} = \refc_n (Z_t, Z'_{n,t})\),
\(\sync_{n,t} = \sync_n (Z_t, Z'_{n,t})\),
\(\delta X_{n,t} = X_t - X'_{n,t}\),
\(\delta V_{n,t} = V_t - V'_{n,t}\),
\(\delta Q_{n,t} = Q_t - Q'_{n,t}\),
\(\delta g_{n,t} = g(Z_t) - g(Z'_{n,t})\),
\(\dd W_{n,t} = e_{n,t}^{\mathrm T} \dd B^\refc_{n,t}\)
and \(r_{n,t} = r(Z_t,Z'_{n,t}) = \theta |\delta X_{n,t}| + |\delta Q_{n,t}|\).
In the following we will omit the subscript \(n\)
in the variables defined above to simplify the notation
and recover it if necessary.
By our construction of Brownian motions,
the original Brownian \(B\) admits the decomposition
\[
\dd B_t = \refc(Z_t, Z'_t) \dd B^\refc_t + \sync(Z_t, Z'_t) \dd B^\sync_t\,.
\]
Therefore, by taking the difference between the two systems of equations
\eqref{eq:kinetic-diffusion} and \eqref{eq:kinetic-diffusion-prime},
we find the difference process \(\delta Z = (\delta X, \delta V)\) satisfy
\begin{align*}
\dd \delta X_t &= \delta V_t \dd t\,, \\
\dd \delta V_t &= ( - \delta V_t - K \delta X_t + \delta g_t) \dd t
+ 2\sqrt 2 \refc_t e_t \dd W_t\,.
\end{align*}
We note that the process \(\alpha\) disappears in the equations above.
Using It\=o's formula, we further obtain
\begin{align*}
\dd (\delta X_t)^2 &= 2 \delta X_t \cdot \delta V_t \dd t\,, \\
\dd (\delta V_t)^2 &= 2 \delta V_t \cdot ( - \delta V_t - K \delta X_t
+ \delta g_t) \dd t
+ 8 (\refc_t)^2 \dd t + 4\sqrt 2 \refc_t \delta V_t \cdot e_t \dd W_t\,, \\
\dd (\delta X_t \cdot \delta V_t) &= |\delta V_t|^2 \dd t
+ \delta X_t \cdot (- \delta V_t -K \delta X_t + \delta g_t) \dd t
+ 2\sqrt 2 \refc_t \delta X_T \cdot e_t \dd W_t\,.
\end{align*}
We have the semimartingle decomposition \(\dd r_t = \dd A^r_t + \dd M^r_t\)
where \(A^r\) is absolutely continuous with
\[
\dd A^r_t \leqslant \bigl( |K| + L_1 - \theta \bigr) |\delta X_t| \dd t
+ \theta |\delta Q_t| \dd t\,,
\]
and \(M^r_t = 2\sqrt 2 \int_0^t \refc_s \dd W_s\) is a martingale.
Consequently, if
\(f : [0, +\infty) \to \mathbb R\)
is a piecewise \(C^2\), non-decreasing and concave function,
then the It\=o--Tanaka formula yields
\(\dd f(r_t) = \dd A^f_t + \dd M^f_t\) where \(A^f\) is absolutely continuous with
\[
\dd A^f_t \leqslant f'_- (r_t) \bigl[ \bigl( |K| + L_1 - \theta \bigr) |\delta X_t| \dd t
+ \theta |\delta Q_t| \bigr] \dd t + 4 f''(r_t) (\refc_t)^2 \dd t\,,
\]
and \(M^f\) is a martingale.

\proofstep{Choice of the Lyapunov function \(G\)}
Define \(G(z, z') = \frac 12 \delta x^{\mathrm T} K \delta x
+ \frac 12 |\delta v|^2
+ \eta \delta x \cdot \delta v\), where we set
\(\eta = \frac 12 \min(1,k)\), and denote \(G_t = G(Z_t, Z'_t)\).
The function \(G\) satisfies the mutual bound
\[
\lambda \bigl(|\delta x|^2 + |\delta v|^2\bigr)
\leqslant G(z, z') \leqslant \frac{\theta}{2} \bigl(|\delta x|^2
+ |\delta v|^2\bigr)\,,
\]
where
\[
\lambda \coloneqq \frac 14 \min (1, k)\,.
\]
We have also
\begin{multline*}
\dd G_t = - \bigl( (1 - \eta) |\delta V_t|^2
+ \eta \delta X_t \cdot \delta V_t
+ \eta \delta X_t^{\mathrm T} K \delta X_t\bigr) \dd t \\
+ (\delta V_t + \eta \delta X_t)\cdot (\delta g_t \dd t
+ 2 \sqrt2 \refc_t e_t \dd W_t)
+ 4 (\refc_t)^2 \dd t \eqqcolon \dd A^G_t + \dd M^G_t\,,
\end{multline*}
where \(A^G\), \(M^G\) are the finite-variation
and the martingale part respectively.
In particular, if \(|\delta X_t| + |\delta V_t| \geqslant R\), then
\begin{multline*}
d A^G_t \leqslant - (\delta V_t,~\delta X_t)
\begin{pmatrix}
1 - L_2 - \eta & -L_2 - (1 + L_2) \eta \\
0 & \eta (K - L_2)
\end{pmatrix}
\begin{pmatrix}
\delta V_t \\
\delta X_t
\end{pmatrix} \dd t + 4 (\refc_t)^2 \dd t\\
\eqqcolon - \delta Z_t^{\mathrm T} M_G \delta Z_t \dd t + 4 (\refc_t)^2 \dd t\,.
\end{multline*}
We choose \(\eta = \frac 12 \min (1, k)\)
and the symmetric part of the matrix \(M_G\) is positive-definite as
its determinant is lower bounded by
\[
(1-L_2-\eta)\cdot \eta(k - L_2) - \frac 14 \bigl( L_2 + (1+L_2)\eta \bigr)^2
\geqslant \frac{\bigl(\min(1,k) - 19L_2\bigr) k}{16}\,.
\]
By the same computation we obtain
\[
M_G \succeq \frac{\bigl(\min(1,k) - 19L_2\bigr) k}{16 \max \bigl(1 - L_2 - \eta, \eta (k - L_2) \bigr)}
\succeq \frac{\bigl(\min(1,k) - 19L_2\bigr) k}{8 \max \bigl(1 - L_2, k - L_2 \bigr)}\,.
\]
As a result,
\begin{align*}
\dd A^G_t - 4(\refc_t)^2 \dd t &\leqslant - \delta Z_t^{\mathrm T} M_G \delta Z_t \dd t
\leqslant - \frac{\bigl(\min(1,k) - 19L_2\bigr) k}{8 \max \bigl(1 - L_2, k - L_2 \bigr)} |\delta Z_t|^2 \dd t \\
&\leqslant - \frac{\bigl(\min(1,k) - 19L_2\bigr) k}{8 \max \bigl(1 - L_2, k - L_2 \bigr) \max\bigl(1, |K| \bigr)} G_t \dd t
\eqqcolon - \kappa_2 G_t \dd t\,.
\end{align*}
To summarize,
if \(|\delta X_t| + |\delta V_t| \geqslant R\),
then \(\dd A^G_t \leqslant - \kappa_2 G_t \dd t + 4 (\refc_t)^2 \dd t\).
In the general case, we have
\begin{align*}
\dd A^G_t &\leqslant \bigl(|\delta V_t|,~|\delta X_t|\bigr)
\begin{pmatrix}
|1 -  \eta| + L_1 & L_1 + (1 + L_1) \eta \\
0 & \eta |K| + L_1
\end{pmatrix}
\begin{pmatrix}
|\delta V_t| \\
|\delta X_t|
\end{pmatrix} \dd t + 4 (\refc_t)^2 \dd t \\
&\leqslant \theta \bigl( |\delta X_t|^2 + |\delta Y_t|^2 \bigl) \dd t
+ 4 (\refc_t)^2 \dd t\,.
\end{align*}

\proofstep{Choice of \(\rho\) and \(f\)}
Now recover the subscript \(n\).
Set
\[
\rho_n(z, z')
= \varepsilon_n G(z, z') + f_n \bigl(\theta |\delta x| + |\delta q|\bigr)
\]
for \(\varepsilon_n > 0\),
and piecewise \(C^2\), non-decreasing and concave
\(f_n : [0, +\infty) \to [0, +\infty)\),
to be determined below.
Denote \(\rho_{n,t} = \rho_n(Z_{n,t}, Z'_{n,t})\).
Then by the previous computations
\(\dd\rho_{n,t} = \dd A^{\rho_n}_t + \dd M^{\rho_n}_t\)
where \(M^{\rho_n}\) is a martingale
and \(A^{\rho_n}\) is absolutely continuous with
\[
\dd A^{\rho_n}_t \leqslant \varepsilon_n \dd A^G_t
+ \bigl[\bigl( |K| + L_1 - \theta \bigr) |\delta X_{n,t}|
+ \theta |\delta Q_{n,t}| \bigr] f_{n,-}' (r_{n,t}) \dd t
+ 4 f''_n(r_{n,t}) (\refc_{n,t})^2 \dd t\,.
\]
Define the functions
\begin{align*}
\varphi (r) &= \exp \biggl( - \frac{\theta r^2}{8} \biggr)\,, \\
\Phi (r) &= \int_0^r \varphi(u) \dd u\,, \\
g_n(r) &= 1 - \frac{\kappa_{n,1}}{2} \int_0^r \Phi(u) \varphi(u)^{-1} \dd u
- \frac{\varepsilon_n}{2} \int_0^r \biggl[ \biggl(1 + \frac{\kappa_1}{2}\biggr) \theta u^2
+ 4 \biggr] \varphi(u)^{-1} \dd u
\end{align*}
for \(r \geqslant 0\),
where \(\kappa_{n,1}\), \(\varepsilon_n\) are positive constants defined by
\begin{align*}
\kappa_{n,1} &= \frac 12 \biggl( \int_0^{r_0 + n^{-1}}
\Phi(u) \varphi(u)^{-1} \dd u \biggr)^{-1}\,, \\
\varepsilon_n &= \frac 12 \biggl(
\int_0^{r_0+n^{-1}} \biggl[ \biggl(1 + \frac{\kappa_{n,1}}{2}\biggr) \theta u^2
+ 4 \biggr] \varphi(u)^{-1} \dd u \biggr)^{-1} \wedge \frac{4}{9R}\,.
\end{align*}
We choose
\[
f_n (r) = \int_0^{r \wedge (r_0 + n^{-1})} \varphi (u) g(u) \dd u\,.
\]
Note that \(\kappa_{n,1}\), \(\varepsilon_n\), \(g_n\) and \(f_n\)
all converge when \(n \to +\infty\), and we denote their limit by
\[
(\kappa_{*,1}, \varepsilon_*, g_*, f_*) =
\lim_{n\to+\infty}(\kappa_1, \varepsilon, g, f)\,.
\]
Denote also
\[
\rho_* (z, z') = \varepsilon_* G(z, z')
+ f_* \bigl(\theta |\delta x|+ |\delta q|\bigr)\,.
\]
Since $G$ is a quadratic form and $f_*'(0) = \varphi(0) g(0) = 1$,
we have the uniform upper limit
\[\limsup_{z \to z'}\frac{\rho_*(z,z')}{|z - z'|}
= \limsup_{z \to z'}\frac{\theta|x - x'| + |x - x' + v - v'|}{\sqrt{|x-x'|^2+|v-v'|^2}}
\leqslant \theta + \sqrt 2 \eqqcolon C_2\,,
\]
validating the last property of the last claim.
By elementary calculations, we also have
\begin{equation}
\label{eq:rho-unif-convergence}
\lim_{n \to +\infty} \sup_{z, z' \in \mathbb R^{2d}}
\frac{|\rho_n(z, z') - \rho_*(z, z')|}{1 + |z - z'|^2} = 0\,.
\end{equation}
The function \(f_n\) is \(C^2\) on \([0,r_0 + n^{-1})\)
and \((r_0 + n^{-1}, +\infty)\), non-decreasing, concave, and satisfies
\[
4f_n'' (r) + \theta f_n'(r) r + \kappa_{n,1} f_n(r)
+ \varepsilon_n \biggl[ \biggl(1 + \frac{\kappa_{n,1}}2\biggr) \theta r^2
+ 4 \biggr] \leqslant 0
\]
for \(r \in [0, r_0 + n^{-1}]\).
Moreover if \(r \in [0, r_0 + n^{-1}]\),
then \(\frac 12 \leqslant g_n (r) \leqslant 1\).
Therefore for every \(r \geqslant 0\), we have
\[
\frac{r}{2} \leqslant \frac{\Phi(r)}{2} \leqslant f_n(r) \leqslant \Phi(r)\,.
\]

\proofstep{Proof of contraction}
We now prove the contraction by investigating the following three cases.
We temporarily omit the subscript \(n\).
\begin{enumerate}
\item
Suppose \(r_t > r_0\).
Then we have \(|\delta X_t| + |\delta V_t| \geqslant R\)
and therefore
\begin{align*}
|\delta X_t|^2 + |\delta V_t|^2
&\geqslant \frac{R^2}{2}
= \frac{R^2}{2f(r_0 + n^{-1})} f(r_0 + n^{-1}) \\
&\geqslant \frac{R^2}{2\Phi(r_0 + n^{-1})} f(r_0 + n^{-1}) \\
&\geqslant \frac{R^2}{2\Phi(r_0 + n^{-1})} f(r_t)\,.
\end{align*}
As a result, \(G_t \geqslant \lambda R^2 \bigl(2\Phi(r_0 + n^{-1})\bigr)^{-1} f(r_t)\).
Hence,
\begin{align*}
\dd A^{\rho}_t &\leqslant - \varepsilon \kappa_2 G_t \dd t
+ \bigl( 4 f''(r_t) + \theta f'_-(r_t) r_t + 4 \varepsilon \bigr)
(\refc_t)^2 \dd t
\leqslant - \varepsilon \kappa_2 G_t \dd t \\
&\leqslant - \biggl( \varepsilon + \frac{2\Phi(r_0 + n^{-1})}{\lambda R^2} \biggr)^{-1}
\varepsilon \kappa_2 \bigl( f(r_t) + \varepsilon G_t\bigr) \dd t \\
&= - \frac{\lambda R^2 \varepsilon}{\lambda R^2 \varepsilon + 2\Phi(r_0 + n^{-1})}
\kappa_2 \rho_t \dd t\,.
\end{align*}

\item
Suppose \(r_t \leqslant r_0\) and \(|\delta Q_t| \geqslant 2n^{-1}\).
Then we have \(\refc_t = 1\) and
therefore
\begin{align*}
\dd A^\rho_t
&\leqslant \varepsilon\theta \bigl( |\delta V_t|^2 + |\delta X_t|^2 \bigr) \dd t
+ \bigl( 4f''(r_t) + \theta f'_- (r_t) r_t + 4 \varepsilon \bigr) \dd t \\
&\leqslant \varepsilon\theta r_t^2 \dd t
+ \bigl( 4f''(r_t) + \theta f'_- (r_t) r_t + 4 \varepsilon \bigr) \dd t \\
&\leqslant \bigl( 4f''(r_t) + \theta f'_- (r_t) r_t
+ \varepsilon (\theta r_t^2 + 4) \bigr) \dd t \\
&\leqslant -\kappa_1 \biggl( f(r_t) + \frac{\varepsilon \theta}{2} r_t^2 \biggr)
\dd t \\
&\leqslant -\kappa_1 \bigl( f(r_t) + \varepsilon G_t \bigr) \dd t
= -\kappa_1 \rho_t \dd t\,.
\end{align*}

\item
Suppose \(r_t \leqslant r_0\) and \(|\delta Q_t| < 2n^{-1}\).
Then we have
\[
|\delta V_t|^2 \leqslant 2 |\delta X_t|^2 + 2|\delta Q_t|^2
\leqslant 2 |\delta X_t|^2 + 8n^{-2}\,.
\]
and
\[
|\delta X_t| \leqslant \theta^{-1} r_t \leqslant \theta^{-1} r_0
= \frac{\theta + 1}{\theta} R \leqslant \frac{3}{2} R\,.
\]
Consequently,
\begin{align*}
\dd A^{\rho}_t &\leqslant \varepsilon\theta
\bigl( |\delta V_t|^2 + |\delta X_t|^2 \bigr) \dd t
- \frac{\theta}{2} |\delta X_t| \dd t + 2\theta n^{-1} \dd t
+ \bigl(4f''(r_t) + 4\varepsilon\bigr) (\refc_t)^2 \dd t \\
&\leqslant \varepsilon \theta \bigl( 3|\delta X_t|^2 + 8n^{-2} \bigr) \dd t
- \frac{\theta}{2} |\delta X_t| \dd t + 2 \theta n^{-1}\dd t \\
&\leqslant - \frac{\theta}{4} |\delta X_t| \dd t
+ (2\theta n^{-2} + 8 \theta \varepsilon n^{-2} ) \dd t\,.
\end{align*}
Since
\begin{align*}
f(r_t) &\leqslant \sup_{r \in (0,r_0]} \frac{f(r)}{r} r_t
= r_t \leqslant \theta |\delta X_t| + 2n^{-1}\,, \\
G_t &\leqslant \theta \bigl( |\delta X_t|^2 + |\delta V_t|^2 \bigr)
\leqslant \theta |\delta X_t|^2 + 4\theta n^{-2}
\leqslant \frac{3\theta R}{2} |\delta X_t| + 4\theta n^{-2}\,,
\end{align*}
we have
\begin{align*}
\dd A^{\rho}_t &\leqslant - \frac{1}{4 + 6\varepsilon R}
\bigl( f(r_t) + \varepsilon G_t - 2n^{-1} - 4\theta \varepsilon n^{-1}\bigr) \dd t
+ (2\theta n^{-2} + 8 \theta \varepsilon n^{-2}) \dd t \\
&= - \frac{\rho_t}{4 + 6\varepsilon R} \dd t + O(n^{-1}) \dd t\,.
\end{align*}
\end{enumerate}
Recovering the subscript \(n\) and combining the three cases above, we obtain
\[
\Expect \bigl[\rho_n (Z_t, Z'_{n,t}) \bigr]
= \Expect [\rho_{n,t}]
\leqslant e^{-\kappa_n t} \Expect [\rho_{n,0}] + O(n^{-1})
= e^{-\kappa_n t} \Expect \bigl[\rho_n (Z_0, Z'_0)\bigr] + O(n^{-1})
\]
for
\begin{align*}
\kappa_n &\coloneqq \min \biggl( \kappa_{n,1},
\frac{\lambda R^2 \varepsilon_n \kappa_2}
{\lambda R^2 \varepsilon_n + 2\Phi(r_0 + n^{-1})},
\frac{1}{4 + 6\varepsilon_n R} \biggr)\,.
\end{align*}
Thanks to the uniform convergence of \(\rho_n\) to \(\rho_*\)
in \eqref{eq:rho-unif-convergence}
and the square-integrability of the processes \(Z_t\), \(Z'_{n,t}\),
we can take the limit \(n \to +\infty\) to derive
\[
\limsup_{n\to+\infty} \Expect \bigl[\rho_* (Z_t, Z'_{n,t})\bigr]
= \limsup_{n\to+\infty} \Expect [\rho_{*,t}]
\leqslant e^{-\kappa_* t} \Expect [\rho_{*,0} ]
= e^{-\kappa_* t} \Expect \bigl[\rho_*(Z_0, Z'_0)\bigr]
\]
for
\begin{align*}
\kappa &\coloneqq \min \biggl( \kappa_{*,1},
\frac{\lambda R^2 \varepsilon_* \kappa_2}{\lambda R^2 \varepsilon_*
+ 2\Phi\bigl((\theta+1)R\bigr)},
\frac{1}{4 + 6\varepsilon_* R} \biggr) \\
&\geqslant \min \biggl( \kappa_{*,1},
\frac{\lambda R^2 \varepsilon_* \kappa_2}{\lambda R^2 \varepsilon_*
+ 2\Phi\bigl((\theta+1)R\bigr)},
\frac{3}{20} \biggr)\,,
\end{align*}
where the inequality is due to \(\varepsilon_* \leqslant \frac{4}{9R}\).
If \(r(z,z') \leqslant r_0\), then
\begin{align*}
|z - z'| &\leqslant \sqrt 2 \bigl( |x - x'| + |v - v'| \bigr) \\
&\leqslant \sqrt 2 \bigl ( 2 |x - x'| + |q - q'| \bigr) \\
&\leqslant \sqrt 2 \, r
\leqslant \frac{2 \sqrt 2\, r_0}{\Phi(r_0)} f_*(r)\,; \\
\intertext{otherwise \(|\delta x| + |\delta v| \geqslant R\), and then}
|z - z'| &= \frac{|\delta x|^2 + |\delta v|^2}{|z - z'|}
\leqslant \frac{\sqrt 2}{\lambda R} G(z, z')\,.
\end{align*}
Therefore for every \(z\), \(z' \in \mathbb R^{2d}\), we have
\[
|z - z'|
\leqslant C_1 \bigl[ \varepsilon_* G(z,z')
+ f_* \bigl(\theta |x-x'| + |q-q'|\bigr) \bigl]
= C_1 \rho_* (z, z')
\]
with
\[
C_1 = \sqrt 2 \max \biggl( \frac{2 (\theta + 1) R}{\Phi\bigl((\theta+1)R\bigr)},
\frac{1}{\lambda \varepsilon_* R} \biggr)\,.
\]
As a consequence,
\[
\limsup_{n\to+\infty}\Expect\bigl[\lvert Z_t - Z'_{n,t}\rvert\bigr]
\leqslant \limsup_{n\to+\infty}\Expect\bigl[\rho_*(Z_t, Z'_{n,t})\bigr]
\leqslant C_1e^{-\kappa t} \Expect \bigl[ \rho(Z_0, Z'_0) \bigr]\,.
\qedhere
\]
\end{proof}

\subsection*{Acknowledgements}

The authors thank Zhenjie Ren for fruitful discussions. The research of P.\,M. is supported by the French ANR grant SWIDIMS (ANR-20-CE40-0022).

\printbibliography

\end{document}